%% file: BBCG.tex
\documentclass[12pt]{article}
\usepackage{amsmath,amstext,amsthm}
\usepackage{amssymb}
\usepackage{amsfonts}
\usepackage{graphicx}
\def\RR{{\bf R}}

\def\HH{{\bf H}}
\def\NN{{\bf N}}
\def\cb{\mathcal{B}}
\def\h{{\mathcal{H}}}
\def\H{{\mathcal{H}^n}}
\def\cR{{\mathcal{R}}}
\def\cC{{\mathcal{C}}}
\def\cCk{{\mathcal{C}_k}}
\def\cRe{{\mathcal{R}_{\varepsilon}}}
\def\cS{{\mathcal{S}}}
\def\cSe{{\mathcal{S}_{\varepsilon}}}
\def\tY{{\tilde{Y}}}
\def\tX{{\tilde{X}}}
\def\tf{{\tilde{f}}}
\def\tg{{\tilde{g}}}
\def\ae{{$\alpha$-ae }}

\def\ra{{\rightarrow}}
\def\Ar{{\Arrowvert}}
\def\lra{{\longrightarrow}}
\def\px{{\partial \tilde{X}}}

\def\dmu{{d\mu_y^c}}
\def\dmut{{d\mu_y^c(\theta)}}

\def\dsit{{d\sigma_y^c(\theta)}}
\def\dnu{{d\nu_y^c}}
\def\dvg{{\mathop{\rm dv_g}\nolimits}}
\def\dvgo{{\mathop{\rm dv_{g_0}}\nolimits}}
\def\dvgt{{\mathop{\rm dv_{\tilde{g}}}\nolimits}}
\def\dh{{\mathop{\rm d_{\mathcal{H}}}\nolimits}}
\def\dgh{{\mathop{\rm d_{\mathcal{GH}}}\nolimits}}

\def\jac{\mathop{\rm Jac}\nolimits}
\def\dim{\mathop{\rm dim}\nolimits}
\def\degree{\mathop{\rm deg}\nolimits}

\def\vol{\mathop{\rm vol}\nolimits}

\def\tr{\mathop{\rm trace}\nolimits}

\def\min{\mathop{\rm min}\nolimits}

\def\ric{\mathrm{Ric}}

\def\scal{\mathrm{Scal}}
\def\deg{\mathrm{deg}}
\def\minvol{\mathrm{minvol}}
\def\K{\mathrm{K}}

\newtheorem{theorem}{Theorem}[section]
\newtheorem{proposition}[theorem]{Proposition}
\newtheorem{definition}[theorem]{Definition}
\newtheorem{corollary}[theorem]{Corollary}
\newtheorem{lemma}[theorem]{Lemma}
\newtheorem{remark}[theorem]{Remark}
\title{Differentiable Rigidity under Ricci curvature lower bound}
\author{L.~Bessi\`eres, G.~Besson, G.~Courtois and S.~Gallot}
\begin{document}

\maketitle
\begin{abstract}
In this article we prove a differentiable rigidity result. Let $(Y, g)$ and $(X, g_0)$ be two closed $n$-dimensional Riemannian manifolds ($n\geqslant 3$) and $f:Y\to X$ be a continuous map of degree $1$. We furthermore assume that the metric $g_0$ is real hyperbolic and denote by $d$ the diameter of $(X,g_0)$. We show that there exists a number $\varepsilon:=\varepsilon (n, d)>0$ such that if the Ricci curvature of the metric $g$ is bounded below by $-n(n-1)$ and its volume satisfies $\vol_g (Y)\leqslant (1+\varepsilon) \vol_{g_0} (X)$ then the manifolds are diffeomorphic. The proof relies on Cheeger-Colding's theory of limits of Riemannian manifolds under lower Ricci curvature bound.
\end{abstract}

\section{Introduction}
Let $Y$ and $X$ be two closed manifolds.
The manifold $Y$ is said to \emph{dominate} $X$ if there is a continuous map $f: Y \rightarrow X$ of degree one.
An $n$-dimensional hyperbolic manifold $X$ has the smallest volume among the set of all Riemannian manifolds
$(Y,g)$ such that $Y$ dominates $X$ and the metric $g$ has Ricci curvature $\ric_g \geq -(n-1)g$.
In dimension $n=2$ this is a consequence of the Gauss-Bonnet formula and in dimension $n\geq 3$ 
this follows from the

\begin{theorem}\label{minvol}\cite{BCG}
Let $(X,g_0)$ be an $n$-dimensional closed hyperbolic manifold and
$Y$ a closed manifold which dominates $X$. Then, for any metric $g$ on $Y$ such that
$\ric_g  \geq   -(n-1)g$, one has
$\vol_g(Y) \geq \vol_{g_0}(X)$, and equality happens if and only if 
$(Y,g)$ and $(X,g_0)$ are isometric.
\end{theorem}  

The minimal volume of a closed manifold $Y$ is defined as 
$$
\minvol (Y) = \inf \left\{ \vol_g (Y) \, / \quad |\K _g| \leq 1 \right\}
$$
where $\K _g$ 
is the sectional curvature of the Riemannian metric $g$.
An $n$-dimensional hyperbolic manifold $X$ is characterized by its minimal volume
among the set of all Riemannian manifolds
$Y$ such that $Y$ is homotopy equivalent to $X$. Namely,

\begin{theorem}\label{bminvol}\cite{be}
Let $X$ be an $n$-dimensional closed hyperbolic manifold and
$Y$ a closed manifold which dominates $X$. Then, $\minvol (Y) = \minvol (X)$ if and only if $X$ and $Y$ are
diffeomorphic.
\end{theorem}

The aim of this paper is to show the following gap result. It improves the above theorem \ref{bminvol} since we now require a lower bound on the Ricci curvature instead of a pinching of the sectional curvature; moreover, under the hypothesis,  we prove that if the volume of $Y$ is close  to the volume of $X$ then these two manifolds are diffeomorphic. More precisely,

\begin{theorem}\label{main-theorem}
Given any integer $n\geq 3$ and $d> 0$, there exists $\varepsilon(n,d)>0$ 
such that the following holds. 
Suppose that $(X,g_0)$ is an $n$-dimensional closed hyperbolic manifold with diameter 
$\leq d$  and that $Y$ is a closed manifold which dominates $X$. 
Then $Y$ has a metric $g$ such that
\begin{eqnarray}
\ric_g & \geq &  -(n-1)g  \label{e0.1}\\
\vol_g(Y)& \leq & (1+\varepsilon)\vol_{g_0}(X) \label{e0.2}
\end{eqnarray}  
if and only if $f$ is homotopic to a diffeomorphism. 
\end{theorem}

In \cite{Far-Jon} the authors prove the existence of closed n-dimensional manifolds $Y$ which are homeomorphic to a closed n-dimensional hyperbolic manifold $(X, g_0)$ but not diffeomorphic to it. An immediate corollary of the above theorem is the following.

\begin{corollary}
With the above notations, there exists $\varepsilon>0$ depending on $n$ and on the diameter of $X$ with the property that for any such $Y$ and any Riemannian metric $g$ on $Y$ whose Ricci curvature is bounded below by $-(n-1)$ one has,
$$\vol (Y, g)> (1+\varepsilon )\vol (X, g_0)\,.$$
\end{corollary} 

To be more precise in \cite{Far-Jon}  the manifold $Y$ is obtained 	as follows:
$$Y=X\sharp \Sigma\,,$$
where $\Sigma$ is an exotic sphere. Not every closed hyperbolic manifold $X$ gives rise to such a $Y$ that is (obviously) homeomorphic but not diffeomorphic to $X$. Indeed, we may have to take a finite cover of $X$. But when we get one construction that works, it does on any finite cover $\overline{X}$ of $X$ as well.  The authors also prove that by taking covers of arbitrary large degree we can put on $Y$ a metric whose sectional curvature is arbitrarily pinched around, say $-1$. The stronger the pinching, the larger the degree. Now assume that $\varepsilon$ could be taken independent of the diameter of $X$; applying the results of \cite{BCG} one could show that the volumes of the two manifold are very close when the pinching on $Y$ is very sharp (close to $-1$). The volume of $Y$ endowed with this pinched metric could then be taken smaller than $(1+\varepsilon)\vol (X, g_0)$, by choosing a covering of large degree; the manifolds though are not diffeomorphic. This gives a contradiction and shows that "size" of $X$ has to be involved in the statement of the theorem, for example its diameter.   

\thanks{This work was supported by the grant ANR: ANR-07-BLAN-0251.}

\subsection{Sketch of the Proof }
We argue by contradiction. Suppose that
there is a sequence $(X_k)_{k \in \NN}$ of closed hyperbolic manifolds with 
diameter $\leq d$ and a sequence of closed manifolds 
$Y_k$, of degree one continuous maps $f_k:Y_k \ra X_k$ and 
metrics $g_k$ on $Y_k$ satisfying the hypothesis (\ref{e0.1}) and 
(\ref{e0.2}) for some $\varepsilon_k$ going to zero. Since $f_k$ is of degree one and 
$X_k$ is hyperbolic, it is equivalent to say (thanks to Mostow's rigidity Theorem) that $f_k$ is homotopic to a 
diffeomorphism or simply that $X_k$ and $Y_k$ are diffeomorphic. 
We thus assume that $Y_k$ and $X_k$ are not diffeomorphic. One then shows that 
up to a subsequence, for large $k$, 
$Y_k$ is diffeomorphic to a closed manifold $Y$, $X_k$ is diffeomorphic 
to a closed manifold $X$, and $X$ and $Y$ are diffeomorphic. One argues as 
follows: by the classical finiteness results we get the sub-convergence of the sequence $\{ X_k\}$.
 Indeed, the 
curvature is $-1$, the diameter is bounded by hypothesis, and there is 
a universal lower bound for the volume of any closed hyperbolic manifold 
of a given dimension, thanks to Margulis' Lemma (see \cite{Bur-Zal}).  
Cheeger's finiteness theorem then applies. Moreover,  
on a closed manifold of dimension $\geq 3$, there is at most one hyperbolic 
metric, up to isometry. We can therefore suppose that $X_k=X$ is a fixed 
hyperbolic manifold. 
The inequality proved in theorem \ref{minvol} provides a lower bound for the volume of $Y_k$ as it 
is explained below. We have no a priori bounds on the diameter 
of $(Y_k,g_k)$, but we can use Cheeger-Colding's theory to obtain 
sub-convergence in the pointed Gromov-Hausdorff topology to a complete metric 
space $(Z,d)$ with small singular set. To obtain more geometric control, 
the idea is to use the natural maps between $Y_k$ and $X$ (see \cite{BCG}). One can show that they sub-convergence to 
a limit map between $Z$ and $X$, which is an isometry. Then $X$ is an
$n$-dimensional smooth closed Riemannian manifold
which is the Gromov-Hausdorff limit of
the sequence $(Y_k,g_k)$ of Riemannian manifold of dimension $n$ satisfying the lower bound 
(\ref{e0.1}) on Ricci curvature, therefore 
$X$ and $Y_k$ are diffeomorphic for large $k$ by a theorem of J. Cheeger and T. Colding.

The paper is organised as follows. 
The construction and the properties of the natural maps are given in Section 2. In Section 3, we construct the limit space $Z$ and the limit map $F:Z \ra X$. 
In Section 4, we prove that $F$ is an isometry and conclude. 

\subsection{Maps of arbitrary degree, scalar curvature}

For two closed manifolds $Y$ and $X$ we said above that $Y$ dominates $X$ if there exists a map of degree one from $Y$ onto $X$. We could have required that there exists a map $f:Y \to X$ of non-zero degree. The main theorem of \cite{BCG} was stated and proved in this set up. More precisely, the following statement holds

\begin{theorem}\cite{BCG}
Let $(X,g_0)$ be an $n$-dimensional closed hyperbolic manifold and
$Y$ a closed manifold such that there exists a map $f: Y\to X$ with non-zero degree denoted deg(f). Then, for any metric $g$ on $Y$ such that
$\ric_g  \geq   -(n-1)g$, one has
$\vol_g(Y) \geq |\deg (f)| \vol_{g_0}(X)$, and equality happens if and only if $f$ is homotopic to a Riemannian 
covering (\textsl{i.e.} locally isometric) of degree |deg(f)| from $(Y,g)$ onto $(X,g_0)$.
\end{theorem}

With the technique developed  in this article,  the following result can be proved

\begin{theorem}\label{degree}
Given any integer $n\geq 3$ and $d> 0$, there exists $\varepsilon(n,d)>0$ 
such that the following holds. 
Suppose that $(X,g_0)$ is an $n$-dimensional closed hyperbolic manifold with diameter 
$\leq d$  and that $Y$ is a closed manifold such that there exists a map $f: Y\to X$ with non-zero degree. 
Then $Y$ has a metric $g$ such that
\begin{eqnarray}
\ric_g & \geq &  -(n-1)g  \label{e0.1}\\
\vol_g(Y)& \leq & (1+\varepsilon)|\deg (f)|\vol_{g_0}(X) \label{e0.2}
\end{eqnarray}  
if and only if $f$ is homotopic to a covering of degree $|\deg (f)|$. 
\end{theorem}

The proof is essentially the one described above; it uses the technique described below and the treatment of an arbitrary degree given in \cite{be}. The fact that the degree can be, in absolute value, greater than one yields extra technicalities. For the sake of clarity we shall omit this proof in the present article and  leave it to the reader. A corollary is,

\begin{corollary}
Let $(X, g_0)$ be a closed $n$-dimensional hyperbolic manifold, then there exists $\varepsilon>0$, such that, for any metric $g$ on the connected
sum $X\sharp X$ satisfying that its Ricci curvature of $g$ is not smaller than $-(n-1)$,
$$\vol (X\sharp X, g) \geq 2(1+\varepsilon )\vol (X, g_0)\,.$$  
\end{corollary}

We may now ask whether such a result could be true with a lower bound on the scalar curvature instead of a lower bound on the Ricci curvature. The situation in dimension $3$, completely clarified by Perelman's work, shows that the answer to this question is negative. More precisely, if $(X,g_0)$ is a $3$-dimensional closed hyperbolic manifolds, a consequence of  \cite[Inequality 2.10]{And} is that, 
$$\inf\{\vol (X\sharp X, g)/\,\, \scal (g)\geq -6\}= 2\vol (X, g_0)\,.$$
In dimension greater or equal to $4$, it follows from \cite{Kob} and the solution to the Yamabe problem that,
$$\inf\{\vol (X\sharp X, g)/\,\, \scal (g)\geq -6\}\leqslant 2\vol (X, g_0)\,.$$

\section{Some a priori control on $(Y,g)$}

Some a priori control on the metric $g$ will be needed in section 2 and 3. We give here the necessary results. 

Let $(X,g_{0})$ be an hyperbolic manifold and $Y$ be a manifold satisfying  the assumptions of  Theorem 
\ref{main-theorem}. For any riemannian metric $g$ on $Y$ satisfying the curvature assumption (\ref{e0.1}), one has the following inequality
\begin{equation}
\vol_g(Y) \geq \vol_{g_0}(X) \,.\label{e1.2}
\end{equation}
It is a consequence of Besson-Courtois-Gallot's inequality (see \cite{BCG}) 
\begin{equation}
 h(g)^n \vol_g(Y) \geq h(g_0)^n \vol_{g_0}(X)\,, \label{ineq-BCG}
\end{equation}
where $h(g)$ is the volume entropy, or the critical exponent, of the metric 
$g$, \textsl{i.e.}: 
$$ h(g)=\lim_{R\ra +\infty} \frac{1}{R}\ln(\vol_{\tg}(B_{\tg}(x,R)))\,,$$
where $\tg$ is the lifted metric on $\tY$. Indeed, any metric $g$ on $Y$ which satisfies (\ref{e0.1}), verifies, by  Bishop's Theorem,
\begin{equation}
h(g) \leq h(g_0) = n-1\,. \label{e1.3}
\end{equation} 

One can obtain a lower bound of the volume of some balls by Gromov's isolation 
Theorem (see \cite[Theorem 0.5]{Gro2}). It shows that  
if the simplicial volume  $||Y||$ -- a topological invariant also called Gromov's norm-- of $Y$  is non-zero, then for any riemannian metric $g$ on $Y$ 
satisfying the curvature assumption (\ref{e0.1}), there exists at least one point 
$y_g \in Y$ such that
\begin{equation}
\vol_g(B(y_g,1)) \geq v_n>0 \label{e1.1}.
\end{equation}
Here $B(y_g,1)$ is the geodesic ball of radius $1$ for the metric $g$ and $v_n$ 
is a universal constant. This theorem applies in our situation since, by an elementary property of the simplicial volume, $|| Y|| \geq ||X||$ if there is a degree one map from $Y$ to $X$ (see \cite{Gro2}). On the other hand,
 $X$ has an hyperbolic metric and hence $ ||X|| >0$ by Gromov-Thurston's Theorem (see \cite{Gro2}). 
 
 Given this universal lower bound for the volume of a unit ball $B(y_{g},1)$, the volume of 
 any ball $B(y,r)$ is bounded from bellow in terms of $r$ and $d(y_{g},y)$. Indeed, recall that under the curvature assumption 
 (\ref{e0.1}),  Bishop-Gromov's Theorem shows that  
 for any $0 < r \leq R$, one has 
 \begin{equation}
 \frac{\vol_{g}(B(y,r))}{\vol_g(B(y,R))} \geq \frac{\vol_{\HH^n}(\rm{B}_{\HH^n}(r))}{\vol_{\HH^n}(\rm{B}_{\HH^n}R))} \label{BG},
 \end{equation}
 where $\rm{B}_{\HH^n}(r)$ is a ball of radius $r$ in the hyperbolic space $\HH^n$. 
 As $B(y_{g},1) \subset B(y,1+d(y_{g},y)+r)$, one deduces from (\ref{BG}) that  
 \begin{eqnarray}
  \vol_{g}(B(y,r)) & \geq & \vol_{g}\left(B(y,1+d(y_{g},y)+r\right)) \frac{\vol_{\HH^n}(\rm{B}_{\HH^n}(r))}{\vol_{\HH^n}(\rm{B}_{\HH^n}(1+d(y_{g},y)+r))}  \\
 & \geq & v_{n}\frac{\vol_{\HH^n}(\rm{B}_{\HH^n}(r))}{\vol_{\HH^n}(\rm{B}_{\HH^n}(1+d(y_{g},y)+r))} \label{loc-vol}\,.
\end{eqnarray}

The curvature assumption (\ref{e0.1}) and the volume estimates (\ref{BG}) or  (\ref{loc-vol}) are those required to use the non-collapsing part of Cheeger-Colding's Theory, as we shall see in section 3.  

\section{The natural maps}
In the following sections 2.1 and 2.2 we recall the construction and the main properties of the 
natural maps defined in \cite{BCG} (see also \cite{BCG2}).
\subsection{Construction of the natural maps}

Suppose that $(Y,g)$ and  $(X,g_0)$ are closed riemannian manifolds and that
$$ f: Y \ra X\,,$$
is a continuous map of degree one.  For the sake of simplicity, we assume that $g_0$ is
hyperbolic (the construction holds in a much more general situation).
 Then, for any $ c > h(g) $ there exists  a $C^1$ map 
 $$F_{c}:Y\longrightarrow X\,,$$
 homotopic to $f$, such that for all $y\in Y$,
\begin{equation}
 |\jac F_{c}(y)| \leq \left(\frac{c}{h(g_{0})}\right)^n \,,\label{basic-inequality}
\end{equation} 
with equality for some $y \in Y$ if and only if $d_{y}F_{c}$ is an homothety of ratio $\frac{c}{h(g_{0})}$. 

Inequality (\ref{ineq-BCG}) is then easily obtained by integration of (\ref{basic-inequality})  and by taking a limit when $c$ goes to $h(g)$. To obtain global 
rigidity properties, one has in general to study carefully the behaviour of 
$F_{c}$ as $c$ goes to $h(g)$.

The construction of the maps is divided in four steps. 
Let $\tY$ and $\tX$ be the universal coverings of $Y$ and $X$ respectively, 
and $\tf: \tY \ra \tX$ a lift of $f$. 

{\bf Step 1:} 
For each $y\in \tY$ and $c>h(g)$, let $\nu_{y}^c$ be the finite measure on $\tY$ defined by  
$$ d \nu_{y}^c(z) = e^{-c.\rho(y,z)}\dvgt(z)$$
where $z \in \tY$, $\tg$ is the lifted metric on $\tY$ and $\rho(.,.)$ is the 
distance function of $(\tY,\tg)$.
 
{\bf Step 2:}  
Fushing forward this measure gives a finite measure $\tf_{*}\nu_{y}^c$ on $\tX$. 
Let us recall that it is defined by 
$$ \tf_{*}\nu_{y}^c(U) = \nu_{y}^c(\tf^{-1}(U)).$$

{\bf Step 3:} One defines a finite measure $\mu_y^c$ on $\px$  by convolution of $ \tf_{*}\nu_{y}$ with all visual probability measures $P_{x}$ of $\tX$. Recall that the visual probability measure $P_x$ at $x \in \tX$ is defined as follows: the unit tangent sphere at $x$ noted $U_{x}\tX$ projects onto the geometric boundary $\px$ by the map 
$$ v \in U_{x}\tX \overset{E_{x}}{\longrightarrow} \gamma_{v}(\infty) \in \px,$$
where $\gamma_{v}(t)=exp_{x}(tv)$. The measure $P_{x}$ is then the push-forward by $E_{x}$ of the canonical probability measure on $U_{x}\tX$, \textsl{i.e.}, for a Borel set $A\in \px$,  $P_{x}(A)$ is the measure of the set of vectors $v \in U_{x}\tX$ such that 
$\gamma_{v}(+\infty)\in A$. 
 
 Then 
\begin{eqnarray*}
\mu_y^c(A) &=& \int_{\tX} P_{x}(A) \ d \tf_{*}\nu_{y}^c(x)\\
        & = & \int_{\tY} P_{\tf(z)}(A)\ d \nu_{y}^c(z).
\end{eqnarray*}
One can identifies $\px$ with the unit sphere in $\textbf{R}^n$, by choosing an origin $o\in \tX$ and using $E_{0}$.
The density of this measure is given by  (see \cite{BCG})
$$ \dmut = 
\left(\int_{\tY} e^{-h(g_{0})B(\tf(z),\theta)}e^{-c\rho(y,z)}\dvgt(z)\right) d\theta,$$
where $\theta\in\px$, $d\theta$ is the canonical probability measure on $S^{n-1}$ and $B(. ,\theta)$ is a Busemann function on $\tX$ normalised to vanish at $x=o$.
We will use the notation
$$ p(x,\theta)=e^{-h(g_{0})B(x,\theta)}.$$

{\bf Step 4: } The map $$ F_c : \tY \lra \tX $$ associates to any $y \in \tY$ the unique $x \in \tX$ which minimizes on $\tX$ the function $$x\to \cb(x) = \int_{\px} B(x,\theta)\ \dmut. $$
(see Appendix A in \cite{BCG}).

The  maps $F_c$ are shown to be $\mathcal{C}^1$ and  equivariant with respect to the actions of the fundamental groups of $Y$ and $X$ on their respective universal cover. The quotient maps, which are also   
 denoted by $F_c : Y \rightarrow X$, are homotopic to $f$.  Note that $F_{c}$ depends heavily on the metric $g$.

\subsection{Some technical lemmas}

Let us give some definitions. 

\begin{definition} 
For $y\in \tY$ let $\sigma_y^c$ be the probability measure on $\px$ defined by 
$$ \sigma_y^c = \frac{\mu_y^c}{\mu_y^c(\px)}.$$
\end{definition}
Let us remark that we have 
$$||\mu_y^c ||= \mu_y^c(\px) = \int_{\tY} e^{-c\rho(y,z)}\dvgt(z)
 =||\nu_y^c ||. $$ 
We consider two positive definite bilinear forms of trace equal to one and the corresponding symmetric endomorphisms. 
\begin{definition}
For any $y \in \tY$, $u,v \in T_{F_c(y)}\tX$, 
$$ \mathbf{h_y^c}(u,v) = \int_{\px} dB_{(F_c(y),\theta)}(u)dB_{(F_c(y),\theta)}(v)
\ \dsit = g_0(\mathbf{H_y^c}(u),v).$$
And, for any $y \in \tY$, $u,v \in T_y\tY$, 
$$ \mathbf{{h'_y}^c}(u,v) = \frac{1}{\mu_y^c(\px)}\int_\tY d\rho_{(y,z)}(u)d\rho_{(y,z)}(v)\ 
\dnu(z)= g(\mathbf{{H'_y}^c}(u),v).$$
\end{definition}



\begin{lemma}
\label{l2.1}
For any $y \in \tY$, $u \in T_y\tY$, $v \in T_{F(y)}\tX$, one has 
\begin{equation}
\left| g_0((I-H_y^c)d_yF_c(u),v)\right| 
\leq c \left( g_0(H_y^c(v),v)\right)^{1/2} \left(g({H'_y}^c(u),u)\right)^{1/2}\,.  \label{e2.4}
\end{equation}  
\end{lemma}

\begin{proof} 
Since $F_c(y)$ is an extremum of the function $\cb$, one has  
\begin{equation}
 d_{F_c(y)}\cb (v) = \int_{\px} dB_{(F_c(y),\theta)}(v)\  \dmut =0 \label{e2.3}
\end{equation}
for each $v \in T_{F_c(y)}\tX$. By differentiating this equation in a direction $u \in T_y\tY$, one obtains
\begin{multline*}
\int_{\px} DdB_{(F_c(y),\theta)}(d_yF_c(u),v) \dmut + \dots  \\
\dots +\int_{\px} dB_{(F_c(y),\theta)}(v)\left(\int_{\tY} 
p(\tf(z),\theta) (-cd\rho_{(y,z)}(u))\dnu(z)\right)\ d\theta = 0
\end{multline*}
Using Cauchy-Schwarz inequality in the second term, one gets
\begin{multline*}
 \left| \int_{\px} DdB_{(F_c(y),\theta)}(d_yF_c(u),v) \dmut \right| \leq    \\
 \int_{\px}| dB_{(F_c(y),\theta)}(v)| 
\left(\int_{\tY} p(\tf(z),\theta) \dnu(z)\right)^{1/2}
 \left(  \int_{\tY} p(\tf(z),\theta)|cd\rho_{(y,z)}(u)|^2  \dnu(z)\right)^{1/2}\ d\theta 
\end{multline*}
which is, using Cauchy-Schwarz inequality again 
\begin{eqnarray*}
 &\leq & c \left(\int_{\px}| dB_{(F_c(y),\theta)}(v)|^2 \int_{\tY} p(\tf(z),\theta) \dnu(z)d\theta \right)^{1/2}
\left(\int_{\px}\int_{\tY} p(\tf(z),\theta)|d\rho_{(y,z)}(u)|^2  \dnu(z)d\theta \right)^{1/2}\\
& = &c \left(\int_{\px}| dB_{(F_c(y),\theta)}(v)|^2 \dmut \right)^{1/2}\left( \int_{\tY} |d\rho_{(y,z)}(u)|^2  \ \dnu(z)\right)^{1/2}\\
& = & c \mu_{y}^c(\partial \tX)
\left(g_0(H_y^c(v),v)\right)^{1/2} \left(g({H'_y}^c(u),u)\right)^{1/2}
\end{eqnarray*}

It is shown in \cite[Chapter 5]{BCG}  that $DdB = g_0 - dB \otimes dB$ for an hyperbolic metric. The left term of the inequality is thus 
$\mu_{y}^c(\partial \tX)g_0((I-H_y^c)d_yF_c(u),v)$. This proves the lemma. 
\end{proof}

\begin{definition}\label{lambda}
Let $0 < \lambda_1^c(y)\leq ...  \leq \lambda_n^c(y) < 1$ be the eigenvalues of $H_y^c$. 
\end{definition}

\begin{proposition}
\label{l2.2} There exists a constant $A:=A(n)>0$ such that, for any $y \in Y$,
\begin{equation}
\left| \jac F_c(y)\right| \leq \left(\frac{c}{h(g_0)}\right)^n 
\left(1-A \sum_{i=1}^n (\lambda_i^c(y)-\frac{1}{n})^2\right) \label{e2.5}
\end{equation}
\end{proposition}

\begin{proof}

The proof is based on the two following lemmas.

\begin{lemma}\label{det-inequality} At each $y \in \tY$, 
        $$|\jac F_c(y)| \leq \left( \frac{c}{\sqrt{n}}\right)^n \frac{\det({H_y}^c)^{1/2}}{\det(I-{H_y}^c)}\,.$$
\end{lemma}
\begin{proof}[Proof of lemma \ref{det-inequality}] Let $\{ v_{i}\}$ be an orthonormal basis of $T_{F_{c}(y)}\tX$ which diagonalizes 
${H_{y}}^c$. We can assume 
that $d_{y}F_{c}$ is invertible otherwise the above inequality is obvious. Let $u'_{i}=\left[(I-{H_y}^c)\circ d_{y}F_{c}\right]^{-1}(v_{i})$. The Schmidt orthonormalisation process  applied to $(u'_{i})$ gives an orthonormal basis $(u_{i})$ at $T_{y}\tY$. The matrix of $ (I-{H_y}^c)\circ d_{y}F_{c}$ 
in the basis $(u_{i})$ and $(v_{i})$ is upper triangular, then
$$ \det(I-{H_y}^c) \jac F_{c}(y) = \prod_{i=1}^{n} g_{0}((I-{H_y}^c)\circ d_{y}F_{c}(u_{i}),v_{i})\,,$$
which gives, with (\ref{e2.4}), 
\begin{eqnarray*}
\det(I-{H_y}^c) |\jac F_{c}(y)| &\leq & c^n \left(\prod_{i=1}^{n} g_{0}({H_{y}}^c (v_{i}),v_{i})\right)^{1/2} 
 \left(\prod_{i=1}^n g({H'_y}^c(u_{i}),u_{i})\right)^{1/2} \\
 & \leq & c^n \det({H_y}^c)^{1/2} \left[\frac{1}{n} \sum_{i=1}^n g({H'_y}^c(u_{i}),u_{i})\right]^{n/2}\,,
 \end{eqnarray*}
 this proves the desired inequality since $\tr({H'_{y}}^c)=1$. 
\end{proof}
\begin{lemma}\label{linear-algebra} Let $H$ a symmetric positive definite $n\times n$ matrix whose trace is equal to one then, if $n\geq 3$, 
$$\frac{\det(H^{1/2})}{\det(I-H)} \leq  \left(\frac{n}{h(g_{0})^2}\right)^{n/2} \left(1-A \sum_{i=1}^n (\lambda_i-\frac{1}{n})^2\right)$$ 
for some positive constant $A(n)$.
\end{lemma}

\begin{proof}[Proof of lemma \ref{linear-algebra}] 
The proof is given in Appendix B5 of \cite{BCG}. 
This is the point where  the rigidity of the natural maps fails in dimension 2. This completes the proof of  proposition \ref{l2.2}. 
\end{proof}
\end{proof}

\subsection{Some nice properties}

We now show that when the 
volumes of $(Y,g)$ and $(X,g_{0})$ are close then
the natural maps $F_c$ have nice properties. In this section, we shall consider 
$F_c$ as a map from $(Y,g)$ to $(X,g_{0})$. We suppose that  the metric $g$ satisfies the curvature assumption (\ref{e0.1}) and the assumption on its volume (\ref{e0.2}) for some $\varepsilon>0$. Let us introduce some terminology. 

\begin{definition} Let $0 < \alpha < 1$. 
We say that a property holds \ae ($\alpha$-almost everywhere) on a set $A$ if the set $A_+$ 
of points of $A$ where the property holds has relative volume bigger or equal to
 $1-\alpha$, \textsl{i.e.} $\frac{\vol(A_+)}{\vol(A)} \geq 1-\alpha$. 
\end{definition}

 We show that $dF_c$ is $\alpha$-close to be isometric \ae on $Y$ for some 
positive $\alpha(\varepsilon,c)$. Moreover $\alpha(\varepsilon,c) \rightarrow 0 $ as $\varepsilon \rightarrow 0$ and $c \rightarrow h(g)$. On the other hand, given any radius $R>0$, one shows that $||dF_c||$ is uniformly 
bounded on  balls $B(y_g,R)$, provided $c$ is close enough to $h(g)$. Recall that we have a  lower  bound for the volume of $(Y,g)$ but we do not have an  upper bound for its diameter. The key point  
is to show that $H_y^c$ is $\alpha$-close to $\frac{1}{n}Id$ on a set of large volume, 
and is bounded on a ball of fixed radius, with respect to the parameters $\varepsilon, c$. 
 
To estimate from above $c-h(g)$ we introduce a parameter $\delta >0$. We suppose that the volume entropy of $g$ satisfies the  inequalities
 \begin{equation}
 h(g) < c \leq h(g) + \delta.    \label{e2.13}
\end{equation}
 
Observe that (\ref{e1.3}), (\ref{e2.5}) and (\ref{e2.13}) implies that 
\begin{equation}
|\jac F_c(y)| \leq \left(\frac{h(g)+\delta}{h(g_0)}\right)^n \leq 
\left(1+\frac{\delta}{n-1}\right)^{n-1}\,,
\end{equation}
for all $y \in Y$. The map $F_c$ is thus almost volume decreasing. 
On the other hand, as $\vol_g(Y)$ is close to $\vol_{g_0}(X)$, the set in $Y$ where $F_c$ decreases the
volume a lot must have a small measure. Equivalently, $|\jac F_c|$ must be close to 1 in $L^1$ norm. We now give a precise statement. 

\begin{lemma} 
\label{l2.3}
If $\delta$ is small enough, there exists $\alpha_{1}=\alpha_{1}(\varepsilon,\delta)>0$ such that $\alpha_1$-ae on $Y$ one has, 
\begin{equation}
1 - \alpha_{1}   \leq  |\jac F_c(y)|, \label{e2.7}
\end{equation}
 and for all $y \in Y$ one has 
\begin{equation}
|\jac F_c(y)|  \leq  1+\alpha_{1}.
\end{equation}
Moreover, $\alpha_{1}(\varepsilon,\delta) \ra 0$ as $\varepsilon$ and $\delta \ra 0$.  
\end{lemma}
  
\begin{proof} 
Let 
$$\alpha = \max \left( \sqrt{\left(1+\frac{\delta}{n-1}\right)^{n-1}-1},
\sqrt{\varepsilon}\right ).$$ 
Thus $\left(1+\frac{\delta}{n-1}\right)^{n-1} \leq 1 + \alpha^2$ and $\varepsilon \leq \alpha^2$.
In particular, $|\jac F_c(y)| \leq 1+\alpha^2 \leq 1 + \alpha$ for all $y \in Y$, if $\delta$ is small enough so that $\alpha$ is less than $1$ (we also assume that $\varepsilon$ is small). 
 
 
As $F_c$ has degree one, we have 
$$\vol_{g_0}(X) = \int_Y F_c^*(\dvgo ) = \int_Y \jac F_c (y)\dvg(y)$$
Denote by $Y_{\alpha_1}$ the 
set of points $y \in Y$ such that 
$$  |\jac F_c(y)| \geq 1-\alpha. $$
We have
\begin{eqnarray}
\vol_{g_0}(X) & \leq & \int_Y |\jac F_c(y)|\dvg(y)\\
            & = & \int_{Y_{\alpha_1}} |\jac F_c(y)|\dvg(y) +
\int_{Y\setminus Y_{\alpha_1}} |\jac F_c(y)|\dvg(y)\\
       & \leq &(1 + \alpha^2)\vol_g(Y_{\alpha_1}) + (1-\alpha)\vol_g(Y\setminus Y_{\alpha_1})\\
       & = & \vol_g(Y) + \alpha^2 \vol_g(Y_{\alpha_1})-\alpha\vol_g(Y\setminus Y_{\alpha_1})
\end{eqnarray}
Then, using the assumption (\ref{e0.2}) and the inequality (\ref{e1.2}) on the volume, we get
\begin{eqnarray}
 \vol_g(Y\setminus Y_{\alpha_1})& \leq &\frac{\vol_g(Y) - \vol_{g_0}(X)}{\alpha} +
   \alpha \vol_g(Y_{\alpha_1})\\
   & \leq &\left( \frac{\varepsilon}{\alpha} + \alpha \right)\vol_g(Y)\\
   & \leq & 2\alpha\vol_g(Y).
 \end{eqnarray}
 Clearly, $1-2\alpha \leq |\jac F_c(y)| $ on $Y_{\alpha_1}$ and 
 $|\jac F_c(y)|  \leq 1 + 2\alpha$ on $Y$ which proves 
 the lemma with $\alpha_{1}(\varepsilon,\delta)=2\alpha$. 
 \end{proof}



From this lemma, we deduce that $F_c$ is almost injective. 
Indeed, let $x \in X$, one defines $N(F_c,x) \in \NN \cup \{\infty\}$ to be the number 
of preimages of $x$ by $F_c$. As $F_c$ has degree one, one has 
$N(F_c,x) \geq 1$ for all $x \in X$. We then define $X_1 := \{ x \in X, N(F_c, x)=1\}$. 
Observe that $N(F_c,x) \geq 2$ on $X\setminus X_1$. 

\begin{lemma}
\label{l2.7}  There exists $\alpha_{2}=\alpha_{2}(\varepsilon,\delta)>0$ such that 
\begin{equation}
\vol_{g_0}(X_1) \geq (1-\alpha_{2})\vol_{g_0}(X)
\end{equation}
and 
\begin{equation}
\int_{X\setminus X_1} N(F_c,x)\ \dvgo(x) \leq \alpha_{2}(\varepsilon,\delta) \vol_{g_0}(X)\,.
\end{equation}
Moreover, $\alpha_{2}(\varepsilon,\delta) \ra 0$ as $\varepsilon$ and $\delta \ra 0$.
\end{lemma}

In particular, there exists $\alpha'>0$ such that $N(F_c,x)=1$ $\alpha'$-ae on $X$. 

\begin{proof} One defines 
$$\alpha_{2}(\varepsilon,\delta)= 2\left(\left(1+\frac{\delta}{n-1}\right)^n(1+\varepsilon) -1\right).$$
From (\ref{e2.5}) and the area formula (see \cite[3.7]{Mor}), we have 
\begin{eqnarray}
\left(\frac{c}{h(g_0)}\right)^n\vol_g(Y)& \geq & \int_Y |\jac F_c(y)|\ \dvg(y)\\
& = & \int_X N(F_c,x)\ \dvgo(x)\\
& = & \int_{X_1} N(F_c,x)\ \dvgo(x)+\int_{X\setminus X_1} (N(F_c,x)-1+1)\dvgo(x)\\
& = & \vol_{g_0}(X) + \int_{X\setminus X_1}( N(F_c,x)-1)\dvgo(x).
\end{eqnarray}
And
\begin{eqnarray}
\vol_{g_0}(X\setminus X_1) & \leq &  \int_{X\setminus X_1} (N(F_c,x)-1) \dvgo(x)\\
& \leq & \left(\frac{c}{h(g_0)}\right)^n  \vol_g(Y) - \vol_{g_0}(X)\\
& \leq & \left(\left(\frac{c}{h(g_0)}\right)^n (1+\varepsilon)-1\right) \vol_{g_0}(X)\\
& \leq & \frac{\alpha_{2}(\varepsilon,\delta)}{2}\vol_{g_0}(X).
\end{eqnarray}
Thus, since $N(F_c,x)\leq 2(N(F_c,x)-1)$ on $X\setminus X_1$, we get
$$ \vol_{g_{0}}(X\setminus X_{1}) \leq \int_{X\setminus X_1} N(F_c,x)\ \dvgo(x) \leq 
 \alpha_{2}(\varepsilon,\delta)\vol_{g_0}(X),$$ 
 and this proves the lemma. 
 \end{proof}


The following lemma says that $dF_c(y)$ is almost isometric 
at points $y$ where $\jac F_c(y)$ is almost equal to $1$.
\begin{lemma} 
\label{l2.4}
There exists $\alpha_3=\alpha_3(\varepsilon,\delta) > 0$ such that the following holds. 
Let $Y_{\alpha_{1}}$ be the set of points where (\ref{e2.7}) holds, that is
$1-\alpha _1(\varepsilon,\delta) \leq |\jac F_c(y)|$. Let $y$ be a point in
$Y_{\alpha_{1}}$ and $u \in T_yY$, then 
\begin{equation}(1-\alpha_3)\Ar u \Ar_g \leq \Ar d_yF_c(u) \Ar_{g_0} 
\leq (1+\alpha_3) \Ar u \Ar_g \,. \label{e2.8}
\end{equation}
Moreover, $\alpha_3(\varepsilon,\delta) \ra 0$ as $\varepsilon$, $\delta \ra 0$.
\end{lemma}

\begin{proof} 
The inequality (\ref{e2.5}) implies that  for all $y \in Y$
$$ \parallel  H_y^c - \frac{1}{n}\mathrm{Id} \Ar ^2 \leq \frac{1}{A}
\left( 1 - \frac{|\jac F_c(y)|}{\left( 1 +\frac{\delta}{n-1}\right)^n } \right).$$
Let us define 
\begin{equation}
\beta_{1}=\beta_{1}(\varepsilon,\delta) = \frac{1}{A^{1/2}}
\left(1 -\frac{1-\alpha_{1}(\varepsilon,\delta)}{\left( 1 +\frac{\delta}{n-1}\right)^n } \right)^{1/2}. 
\label{beta_1}
\end{equation}

where $\alpha_{1}(\varepsilon,\delta)$ is the constant from Lemma \ref{l2.3}. 
Clearly, $\beta_{1}(\varepsilon,\delta) \rightarrow 0$ as $\varepsilon$ and $\delta \ra 0$. 
Let $Y_{\alpha_1}$ be the set of points where (\ref{e2.7}) holds. On $Y_{\alpha_1}$, 
one has 
\begin{equation}
\parallel  H_y^c - \frac{\mathrm{Id}}{n} \Ar ^2 \leq {\beta_{1}}^2. \label{HId}
\end{equation}
Let  $\{ u_i\}_{i=1,\dots,n}$ be an orthonormal basis 
of $T_yY$ and $v_i = d_yF(u_i)$. Writing 
$\mathrm{Id}-H_y^c=\frac{n-1}{n}\mathrm{Id} 
+ \frac{1}{n}\mathrm{Id} - H_y^c$, one gets 
\begin{align}
\begin{split}
\left| g_0\left((\mathrm{Id}-H_y^c)d_yF_c(u_i),d_yF_c(u_i)\right)\right| \geq &
 \left|g_0\left((\frac{n-1}{n}\mathrm{Id})d_yF_c(u_i),d_yF_c(u_i)\right)\right| \\
 & - 
\left| g_0\left((\frac{1}{n}\mathrm{Id} - H_y^c) d_yF_c(u_i),d_yF_c(u_i)\right)\right| 
\end{split}\\
 & \geq  \frac{n-1}{n} ||d_yF_c(u_i)||_{g_0}^2 - \parallel \frac{1}{n}\mathrm{Id}-H_y^c|| .||d_yF_c(u_i)||_{g_0}^2\\
 &\geq  \left(\frac{n-1}{n}-\beta_1\right)||d_yF_c(u_i)||_{g_0}^2 \label{left}\,.
\end{align} 
Writing $H_y^c= \frac{1}{n}\mathrm{Id} + H_y^c - \frac{1}{n}\mathrm{Id} $, 
one has 
\begin{align}
\begin{split}
 g_0\left(H_y^cd_yF_c(u_i),d_yF_c(u_i)\right)^{1/2} \leq\ & 
 g_0\left((\frac{1}{n}\mathrm{Id})d_yF_c(u_i),d_yF_c(u_i)\right)^{1/2}  \\ 
& + \left| g_0\left((H_y^c - \frac{1}{n}\mathrm{Id}) d_yF_c(u_i),d_yF_c(u_i)\right)\right|^{1/2}
\end{split}\\
& \leq \left(\frac{1}{\sqrt{n}} + \beta_1^{1/2}\right) ||d_yF_c(u_i)||_{g_0}\,.
\end{align}

Taking the trace of the right hand side of (\ref{e2.4}) and using 
the Cauchy-Schwarz inequality, one has 
\begin{align}
\begin{split}
 \sum_{i=1}^n\ g_0\left(H_y^cd_yF_c(u_i),d_yF_c(u_i)\right)^{1/2} g({H'_y}^c(u_i),u_i)^{1/2} &\leq  
\left(\frac{1}{\sqrt{n}} + \beta_1^{1/2}\right) \left(\sum_{i=1}^n ||d_yF_c(u_i)||_{g_0}^2 \right)^{1/2} \\
& \times \left(\sum_{i=1}^n g({H'_y}^c(u_i),u_i) \right)^{1/2}
\end{split}\\
& = \left(\frac{1}{\sqrt{n}} + \beta_1^{1/2}\right) \left(\sum_{i=1}^n ||d_yF_c(u_i)||_{g_0}^2 \right)^{1/2} \label{right}
\end{align}

By (\ref{e2.4}), the trace of (\ref{left}) is not greater than the right hand side of (\ref{right}) multiplied by $c$,  hence 
$$ \left(\frac{n-1}{n}-\beta_1 \right) \sum_{i=1}^n ||d_yF_c(u_i)||_{g_0}^2 
\leq c \left(\frac{1}{\sqrt{n}} + \beta_1^{1/2}\right) \left(\sum_{i=1}^n ||d_yF_c(u_i)||_{g_0}^2 \right)^{1/2}\,,$$
and
$$
\left(\sum_{i=1}^n ||d_yF_c(u_i)||_{g_0}^2 \right)^{1/2} \leq 
c \frac{\frac{1}{\sqrt{n}} + \beta_1^{1/2}}{\frac{n-1}{n}-\beta_1} \leq
\sqrt{n}(1+\frac{\delta}{n-1}) \frac{1 + \sqrt{n}\beta_1^{1/2}}{1 -\frac{n}{n-1}\beta_1}.
$$
Let us define 
$$ \beta_2:=\beta_2(\varepsilon,\delta)= (1+\frac{\delta}{n-1})^2\left( \frac{1 + \sqrt{n}\beta_1^{1/2}}{1 -\frac{n}{n-1}\beta_1}\right)^2-1.$$
Clearly, $\beta_2(\varepsilon,\delta) \ra 0$ as $\varepsilon$ and $\delta \ra 0$. One has 
 $$\sum_{i=1}^n ||d_yF_c(u_i)||_{g_0}^2 \leq n (1 + \beta_2). $$
 
Let $L$ be the endomorphism of $T_yY$ defined by $L = (d_yF_c)^* \circ d_yF_c$. We have 
\begin{equation}
\tr(L) = \sum_{i=1}^n g(L(u_i),u_i)  =  \sum_{i=1}^n g(d_yF_c(u_i),d_{y}F_{c}(u_i))
				 \leq   n (1 + \beta_2).
\end{equation}
On the other hand 
\begin{equation*}
|1-\alpha|^2 \leq |\jac F_c(y)|^2  =  \det(L) \leq \left(\frac{\tr(L)}{n}\right)^n
	   \leq  (1+\beta_2)^n,
\end{equation*}
which shows that there is almost equality in the  arithmetico-geometric inequality. We then get that there exists  some 
$\alpha_3(\varepsilon,\delta) >0$, with $\alpha_3(\varepsilon,\delta)\ra 0$ as  
$\varepsilon,\delta \ra 0$, such that 
$$ ||L - Id|| \leq \alpha_3(\varepsilon,\delta).$$
Thus for any $y \in Y_{\alpha_1}$ and $u \in T_yY$
\begin{equation}
 (1-\alpha_3)\Ar u \Ar \leq \Ar d_yF_c(u) \Ar_{g_0}\leq (1+\alpha_3)\Ar u \Ar
\end{equation}
and $d_yF_c$ is almost isometric. 
\end{proof}


We now prove that given a fixed radius $R>0$, the natural maps $F_{c}$ have uniformly bounded 
differential $dF_{c}$ on 
$B(y_{g},R)$ if the parameters $\varepsilon$, $\delta$ are sufficiently small. Recall that the point
$y_g$ has been chosen such that (\ref{e1.1}) holds, namely $\vol_g(B(y_g, 1)) \geq v_n$.

\begin{lemma} 
\label{l2.5} Let $R>0$, then  there exist  $\varepsilon(R)>0$ and $\delta(R)>0$ such that for any 
$0<\varepsilon<\varepsilon(R)$ and $0<\delta < \delta(R)$, and for any $y \in B(y_g,R)$, 
\begin{equation}
\Ar d_yF_c \Ar \leq 2\sqrt{n}   \label{e2.10}
\end{equation}
\end{lemma}

\begin{proof}  We first prove that  for all $y \in Y$, $\Ar d_yF_c \Ar$ 
is bounded from above by $\lambda_n^c(y)$,  the maximal eigenvalue of $H_y^c$ (see Definition \ref{lambda} ). 
Recall that $0 < \lambda_n^c < 1$. 
Let $u$  be a unit vector in $T_y\tY$ and $v = d_yF_c(u)$.  Equation(\ref{e2.4}) gives 
 
\begin{equation}
\left(1-\lambda_n^c(y)\right) \left|g_0(d_yF_c(u),d_yF_c(u))\right| 
\leq c \lambda_n^c(y)^{1/2} g_0\left(d_yF_c(u),d_yF_c(u)\right)^{1/2}
\end{equation}
hence 
\begin{equation}
 \Ar d_yF_c(u) \Ar_{g_0} \leq \frac{c\sqrt{\lambda_n^c(y)}}{1-\lambda_n^c(y)}.
\label{e2.11}
\end{equation}
We thus have to show that $\lambda_n^c(y)$ is not close to $1$. More precisely, let $\beta>0$ such that 
$\frac{1}{n}+\beta <1$, one then defines 
$$ \gamma(\delta,\beta) := \left( \frac{n-1+\delta}{n-1-n\beta} \right)\sqrt{1+n\beta}-1>0.$$
Clearly, $\gamma(\beta,\delta) \ra 0$ as $\delta,\beta \ra 0$.
 One can check that if $\lambda_n^c(y) \leq \frac{1}{n}+ \beta$, then $\Ar d_yF_c(u) \Ar_{g_0} \leq \sqrt{n}(1+\gamma)$. 
 For our purpose, we may suppose that $\gamma \leq 1$. 
Now let $\delta_{n}>0$ and $\beta_{n}>0$ be such that if $0<\delta \leq 10\delta_{n}$ and  $0<\beta \leq 10\beta_{n}$, then 
$\gamma(\delta,\beta) \leq 1$. Moreover we define  $\varepsilon_{n}>0$ such that 
if $0 < \varepsilon < \varepsilon_{n}$ and $0<\delta \leq 10\delta_{n}$ then, with the notations (\ref{beta_1}) of  
Lemma \ref{l2.4}
 $\beta_{1}(\varepsilon,\delta) \leq \beta_{n}$. In what follows, we suppose $\varepsilon $ and 
$\delta$ sufficiently small.

By (\ref{HId}) we have that  $|\lambda_n^c(y)-\frac{1}{n}| \leq \beta_{1}(\varepsilon,\delta)$ on $Y_{\alpha_1}$. Recall that $Y_{\alpha_1}$ has a large 
relative volume in $Y$. The idea is first to estimate $\lambda_n^c$  on a neighbourhood of $Y_{\alpha_1}$ and then 
to  show that this neighbourhood contains $B(y_g,R)$ if the parameters $\varepsilon$ and $\delta$ are sufficiently small relatively to $R$. 

For this purpose we need to estimate the variation of $\lambda_n^c$. Recall that $H_y^c$ is defined by 
$$ g_0(H_y^c(u),v) = \int_{\px} dB_{(F_c(y),\theta)}(u)dB_{(F_c(y),\theta)}(v)
\ \dsit. $$
Let $U$, $V$ be parallel vector fields near $F_c(y)$ extending 
unit tangent vectors at $F_c(y)$,  $u$ and $v$. We compute 
the derivative of $g_0(H_y^c(U),V) $ in a direction $w \in T_yY$: 
\begin{align*}
\begin{split}
 w.g_0 & (H_y^c(U),V)=\int_{\px}  DdB_{(F_c(y),\theta)}(d_yF(w),U)  dB_{(F_c(y),\theta)}(V)
 \dsit +  
 \end{split}\\
&  
 \int_{\px}dB_{(F_c(y),\theta)}(U) DdB_{(F_c(y),\theta)}(d_yF(w),V)
 \dsit  
+\int_{\px} dB_{(F_c(y),\theta)}(U)dB_{(F_c(y),\theta)}(V)w.\dsit
 \end{align*}
The Buseman functions of the hyperbolic space satisfies $\Ar DdB \Ar \leq 1$ and $\Ar dB \Ar \leq 1$ and thus 
$$|w.g_0 (H_y^c(U),V)| \leq 2 \Ar d_yF_c(w)\Ar_{g_0}+\left| \int_{\px} w.\dsit
 \right|.$$ 
Recall that
$$ \dsit = \frac{\dmu(\theta)}{\mu_y^c(\px)} = 
\frac{\int_{\tY} p(\tf(z),\theta) e^{-c\rho(y,z)}\dvgt(z)\  }
{\int_{\tY} e^{-c\rho(y,z)}\dvgt(z)} d\theta\,.$$ 
Differentiating this formula yields
\begin{eqnarray}
w.\dsit = \frac{\int_{\tY} p(\tf(z),\theta) (-c.d\rho_{(y,z)}(w))
 e^{-c\rho(y,z)}\dvgt(z)\  }{\mu_y^c(\px)}d\theta - \\ 
\frac{\dmu(\theta)}{\mu_y^c(\px)^2}.\int_{\tY}(-c.d\rho_{(y,z)}(w)) 
e^{-c\rho(y,z)}\dvgt(z)\,.
\end{eqnarray}
Since $|d\rho_{(y,z)}(w)| \leq \Ar w \Ar_g$, we have 
\begin{equation}
\left| \int_{\px} w.\dsit \right| \leq 
\int_{\px}2c\Ar w \Ar_g \dsit=2c\Ar w \Ar_g\,,
\end{equation}

we gives that, $\left| w.g_0 (H_y^c(U),V) \right| \leq 2 \Ar d_yF_c(w)\Ar_{g_0} + 2c\Ar w \Ar_g$.
If $w$ is a unit vector, (\ref{e2.11}) yields 
\begin{equation}
\left| w.g_0 (H_y^c(U),V)\right| \leq 2c\left(\frac{\sqrt{\lambda_n^c(y)}}
{1-\lambda_n^c(y)} +1 \right)\,. \label{e2.12}
\end{equation}

Let us now consider small constants $\eta > \beta >0$ and define 
$$ r(\delta,\beta, \eta) := \frac{\eta - \beta}{2(n-1+\delta) \left(\frac{\sqrt{\frac{1}{n} + \eta }} {1-(\frac{1}{n} + \eta )}+1\right)}>0.
$$
Our goal is to prove that 
$$  \inf \Big\{d(y_{0},y_{1})\ | y_{0},y_{1} \in Y,  \lambda_n^c(y_{0}) \leq \frac{1}{n}+\beta,\,\, \lambda_n^c(y_{1}) \geq \frac{1}{n}+\eta \Big\} \geq r(\delta,\beta, \eta).$$ 
Let $y_{0} \in Y$ so that $\lambda_n^c(y_{0}) \leq \frac{1}{n}+\beta$.  Assume that there exists $y \in Y$ such that $\lambda_n^c(y)\geq \frac{1}{n}+\eta$. One defines 
$$ r := \inf \Big\{d(y_{0},y)\ | y \in Y,  \lambda_n^c(y) \geq \frac{1}{n}+\eta \Big\}.$$ 
By continuity, there exists $y_{1} \in Y$ such that $\lambda_n^c(y_{1}) = \frac{1}{n}+\eta$ and $d(y_{0},y_{1}) = r$.\\  
Let $\gamma : [0,r] \longrightarrow Y$ be a minimising geodesic from $y_{0}$ to $y_{1}$. 
We easily see that  $\lambda_n^c(\gamma(t)) < \frac{1}{n} + \eta$ for any $0 \leq t < r$. Let $U(t)$ be a parallel vector 
field in $X$ along $F_c(\gamma)$ such that $U(r)$ is a unit eigenvector of $H_{y_{1}}^c$.   
Then, using (\ref{e2.12}) with $ \dot{\gamma}.g_{0}(H^c_{\gamma(t)}U(t),U(t)) = \frac{d}{dt} g_{0}(H^c_{\gamma(t)}U(t),U(t))$, 
one has 
\begin{eqnarray}
\left| \lambda_n^c(y_{1})- \lambda_n^c(y_{0}) \right|
 & \leq &  \left| g_{0}(H^c_{\gamma(r)}U(r),U(r)) -  g_{0}(H^c_{\gamma(0)}U(0),U(0))\right|\\
 & =  & \left |\int_0^r \frac{d}{dt} g_{0}(H^c_{\gamma(t)}U(t),U(t)) dt \right|\\
 & \leq & 2c \int_0^r (\frac{\sqrt{\lambda_n^c(\gamma(t))}}
{1-\lambda_n^c(\gamma(t))} +1) dt\\
& \leq & 2cr\left( \frac{\sqrt{\frac{1}{n} + \eta }}
{1-(\frac{1}{n} + \eta )}+1 \right)\,.
\end{eqnarray}

As a consequence
$$  r \geq \frac{\eta - \beta}{2(n-1+\delta) \left(\frac{\sqrt{\frac{1}{n} + \eta }} {1-(\frac{1}{n} + \eta )}+1\right)} = r(\delta,\beta, \eta).$$ 

We now set   $\eta = 2\beta_{n}$ so that  $\gamma(\delta,\eta) \leq 1$ for any $\delta\leq \delta_{n}$ . 
One then defines \\$r_{n} :=  r(\delta_{n}, \beta_{n},2\beta_n)$. 
Let us recall that for 
$\varepsilon \leq \varepsilon_{n}$ and $\delta \leq \delta_{n}$, we have $\beta_{1}(\varepsilon,\delta) \leq \beta_{n}$.  
On $Y_{\alpha_1}$, one has $\lambda_{n}^c(y) \leq \frac{1}{n}  +  \beta_{1}(\varepsilon,\delta)\leq  
\frac{1}{n} + \beta_{n}$.  Hence, if $\lambda_{n}^c(y_{1}) \geq \frac{1}{n}  +  2\beta_{n}$, one has 
$$d(y_{1},Y_{\alpha_{1}}) \geq r(\delta,\beta_{1}(\varepsilon,\delta),2\beta_{n}) \geq r(\delta_{n}, \beta_{n},2\beta_n)=r_{n}.$$
We thus have proved that in the 
$r_{n}$-neighbourhood of 
$Y_{\alpha_1}$, one has $\lambda_{n}^c(y) \leq \frac{1}{n} + 2\beta_{n}$.  This implies that 
$$||d_{y}F_{c}|| \leq (1+\gamma(\delta,2\beta_{n})) \sqrt{n} \leq 2 \sqrt{n}.$$ 

Let us denote by $V_{r_{n}}(Y_{\alpha_1})$ the $r_{n}$-neighbourhood of 
$Y_{\alpha_1}$. It remains to show that $B(y_{g},R) \subset V_{r_{n}}(Y_{\alpha_1})$, if $\varepsilon \leq \varepsilon(R)$ 
and $\delta \leq \delta(R)$. 
Let us recall that $\frac{\vol_g(Y_{\alpha_1})}{\vol_g(Y)} \geq 1 - \alpha_1$, hence 
$$\vol_g(Y\setminus Y_{\alpha_1}) \leq \alpha_{1} \vol_{g}(Y) \leq \alpha_{1}(1+\varepsilon) \vol_{g_{0}}(X):=v(\varepsilon,\delta).$$ 
Clearly,  $v(\varepsilon,\delta) 
\ra 0$ when $\varepsilon$, $\delta \ra 0$. 
On the other hand, by (\ref{loc-vol}) for any $y \in B(y_{g},R)$ we have 
\begin{equation}
 \vol_{g}(B_{g}(y,r_{0}))   \geq  v_{n}\frac{  \vol_{\HH^n}(\rm{B}_{\HH^n}(r_{0}))}{\vol_{\HH^n}(\rm{B}_{\HH^n}(1+R+r_{0}))} :=v_{0}(R)>0 .
\end{equation} 
If $v_{0}(R) > v(\varepsilon,\delta) $, then for any $y \in B(y_{g},R)$ one has $B_{g}(y,r_{n}) \not\subset Y\setminus Y_{\alpha_1}$, which means that 
$B_{g}(y,r_{n})$ intersects $Y_{\alpha_1}$. This shows that $d(y,Y_{\alpha_1}) <r_{n}$ and $y \in V_{r_{n}}(Y_{\alpha_1})$.

The lemma is proved if we define $\varepsilon=\varepsilon(R)>0$ and $\delta=\delta(R)>0$ to be sufficiently small constants such that 
$v(\varepsilon,\delta) < v_{0}(R)$.
\end{proof}


We now prove that $F_c$ is almost $1$-lipschitz. 

\begin{lemma}
\label{l2.6}
For any fixed $R>0$, there exists $\varepsilon_{2}(R)>0$ and $\delta_{2}(R)>0$ such that 
for every $0<\varepsilon<\varepsilon_{2}(R)$ and $0<\delta<\delta_{2}(R)$, there exists 
 $\kappa=\kappa(\varepsilon,\delta,R) >0$ such that on $B_g(y_g,R)$:
\begin{equation}
d_{g_0}(F_c(y_1),F_c(y_2)) \leq (1+\kappa) d_g(y_1,y_2) + \kappa\,.
\end{equation}
Moreover, $\kappa(\varepsilon,\delta,R) \ra 0$ as $\varepsilon$, $\delta \ra 0$.
\end{lemma}

\begin{proof} 
The idea goes as follows. We have proved that $d_yF_c$ is almost isometric on $Y_{\alpha_1}$.  
On the other hand, $||d_yF_c||$ is uniformly bounded in $B(y_g,R)$ if the parameters $\varepsilon$ and $\delta$ are chosen sufficiently small. To prove the lemma 
one computes the lengths of $F_c(\gamma)$ where $\gamma$ is a minimising geodesic in 
$B(y_g,R)$ whose intersection with $Y_{\alpha_1}$ is large. Existence of such geodesics 
follows from an integral geometry lemma due to T. Colding.

Fix some $R>0$. We define the following constants : 

If $d>0$,  
$$c_1(n,d) := \sup_{0<s/2<r<s<d} \frac{\vol_{\HH^n}(\partial B_{\HH^n}(s))}{\vol_{\HH^n}(\partial B_{\HH^n}(r))}.$$
If $\tau>0$, $R>0$, 
 $$c_2(n,\tau,R) := c_1(n,2R)(2\tau \vol_{\HH^n}(\rm{B}_{\HH^n}(\tau)).$$
If $\varepsilon>0$, $\delta>0$, 
$$\theta(\varepsilon,\delta) := 2 \alpha_3^2(\varepsilon,\delta)\vol_{g_0}(X) + 2(4n+1)\alpha_{1}(\varepsilon,\delta)\vol_{g_0}(X).$$
Clearly, $\theta(\varepsilon,\delta)\ra 0$ as $\varepsilon, \delta \ra 0$.  
   
Let $\tau(\varepsilon,\delta,R)>0$ be the function implicitely defined by 
    $$ \vol_{\HH^n}(\tau)\tau := \theta(\varepsilon,\delta) \frac{2c_1(n,2R) \vol_{\HH^n}(1+R+1)^2}{v_n^2}.$$
 Again, we easily see that, for fixed $R$, $\tau(\varepsilon, \delta,R) \ra 0$ as $\varepsilon$, $\delta \ra 0$. 
We also choose $\varepsilon_2(R)>0$ and $\delta_{2}(R)>0$ such that $\varepsilon_{2}(R) \leq \varepsilon(2R)$, 
 $\delta_2(R)<\delta(2R)$ and such that, if 
 $0 < \varepsilon \leq \varepsilon_2(R)$ and $0 < \delta < \delta_2(R)$, then 
 $\tau(\varepsilon,\delta,R) <<1$. 

Finally, one defines $\kappa(\varepsilon,\delta,R):= \max(2\sqrt{n}\sqrt{\tau},8\sqrt{\tau})$. From the remarks above we can choose $\varepsilon_{2}(R)$ and $\delta_{2}(R)$ so that $\kappa(\varepsilon,\delta,R)<1/R$ (for $0 < \varepsilon \leq \varepsilon_2(R)$, $0 < \delta < \delta_2(R)$ and $R$ big).

There are two cases.

\textbf{Case i)} Let  $y_1$, $y_2$ in $B_g(y_g,R)$ such that $d(y_1,y_2) \leq \sqrt{\tau}$.
Using (\ref{e2.10}), if $0 < \varepsilon < \varepsilon(2R)$, $0 < \delta < \delta(2R)$ one has  
\begin{equation}
d(F_c(y_1),F_c(y_2)) \leq 2\sqrt{n}\sqrt{\tau} \leq \kappa. \label{e2.19}
\end{equation}

\textbf{Case ii) : } Let $y_1$, $y_2$ in $B_g(y_g,R)$ such that 
 $d(y_1,y_2) \geq \sqrt{\tau}$. We will use the following theorem, due to J. Cheeger and
T. Colding, cf. \cite[Theorem 2.11]{Col} that we describe now in a particular case. We keep the notations
of \cite{Col}.

 Let us define $A_1=B_g(y_1,\tau)$, $A_2 = B_g(y_2,\tau)$ and $W=B_g(y_g,2R)$ where $y_1$ and $y_2$
are points as above sitting on a complete riemannian manifold $(Y,g)$ with
$\ric_g \geq -(n-1)g$. 
 For any $z_1 \in A_1$ and any unit vector $v_1 \in T_{z_1}Y$, the set 
$I(z_1,v_1)$ defined by
 $$ I(z_1,v_1) = \{t \mid \gamma(t) \in A_2,\gamma_{|[0,t]} \text{ is minimal }, \gamma'(0)=v_1 \}$$ 
has a measure $|I(z_1,v_1)|$ bounded above by $2\tau$. Thus 
 $$D(A_1,A_2) :=\sup_{z_1,v_1} |I(z_1,v_1)| \leq 2 \tau,$$
and similarly, $D(A_2,A_1) \leq 2\tau$. 
 For any $z_1 \in A_1$ and $z_2 \in A_2$, let $\gamma_{z_1z_2}$ be a minimizing 
geodesic from $z_1$ to $z_2$. Clearly, $\gamma \subset B(y_g,2R)$. Then, by 
\cite[Theorem 2.11]{Col}, we have for any non negative integrable function $e$ defined on $Y$,

\begin{equation*}
 \int_{A_1 \times A_2} \int_0^{d(z_1,z_2)} e(\gamma_{z_1,z_2})(s)\ ds 
  \leq  c_1(n,2R)\left(D(A_1,A_2)\vol(A_1)+D(A_2,A_1)\vol(A_2)\right) 
\end{equation*}
\begin{equation}
\times \int_W e(y)\ \dvg(y). \label{cheeg-col}
\end{equation}

By  Bishop's Theorem, for $i=1$,$2$ we have  
$$\vol_g(A_i) \leq \vol_{\HH^n}(\rm{B}_{\HH^n}(\tau)),$$ 
and thus 
$$ c_1(n,2R)\left(D(A_1,A_2)\vol(A_1)+D(A_2,A_1)\vol(A_2)\right) \leq c_2(n,\tau,R).$$

Therefore, applying (\ref{cheeg-col}) to the function
$$e(y)= \sup_{u\in U_yY} \left( \Ar d_yF_c(u)\Ar - \Ar u \Ar \right)^2$$
and using (\ref{e2.8}) on $W\cap Y_{\alpha_1}$ and (\ref{e2.10}) on $W \setminus Y_{\alpha_1}$, we get
\begin{eqnarray}
 \int_{A_1 \times A_2} \int_0^{d(z_1,z_2)} e(\gamma_{z_1,z_2})(s)\ ds 
 & \leq & c_2(n,\tau,R) \left(\int_{W\cap Y_{\alpha_1}}e(y)\ \dvg(y) + \int_{W\setminus Y_{\alpha_1}} e(y)\ \dvg(y)\right) \nonumber \\
& \leq & c_2(n,\tau,R) \left(\alpha_3^2.\vol_g(Y) + (4n+1)\vol_g(Y\setminus Y_{\alpha_1})\right) \nonumber \\
& \leq & c_2(n,\tau,R) \theta(\varepsilon,\delta).  \label{e2.15} 
\end{eqnarray}
Now, if we denote by $\gamma := \gamma_{z_1z_2}$, we have 
\begin{eqnarray*}
|\ell(F_c \circ \gamma) - \ell(\gamma)|& = & \left|\int_0^{d(z_1,z_2)} 
\Ar d_{\gamma(s)}F_c(\dot{\gamma}) \Ar - \Ar \dot{\gamma} \Ar \ 
ds \right|\\
 & \leq & \int_0^{d(z_1,z_2)} \sup_{u\in T_yY} \left|\Ar d_{\gamma(s)}F_c(u)\Ar - \Ar u \Ar \right| \ ds.
\end{eqnarray*}
Using Cauchy-Schwarz inequality we have
\begin{eqnarray*}
\frac{|\ell(F_c \circ \gamma) - \ell(\gamma)|^2 }{d(z_1,z_2)} & \leq & 
\frac{\left( \int_0^{d(z_1,z_2)} \sup_u \left|\Ar d_{\gamma(s)}F_c(u)\Ar - \Ar u \Ar \right| \ ds  \right)^2 }{d(z_1,z_2) }\\
& \leq & \int_0^{d(z_1,z_2)} e(\gamma(s)) \ ds.
\end{eqnarray*}

Integrating on $A_1 \times A_2$, we deduce from (\ref{e2.15}) that 
\begin{equation}
 \int_{A_1 \times A_2} \frac{|\ell(F_c \circ 
\gamma_{z_1z_2}) - \ell(\gamma_{z_1z_2})|^2 }{d(z_1,z_2)}\dvg(z_1)\dvg(z_2)
\leq c_2(n,\tau,R) \theta(\varepsilon,\delta)\,. \label{e2.16}
\end{equation}
By (\ref{loc-vol}), for $i=1$,$2$ one has 
$$  \vol_g(A_i) \geq  v_{n}\frac{\vol_{\HH^n}(\rm{B}_{\HH^n}(\tau))}{\vol_{\HH^n}(\rm{B}_{\HH^n}(1+R+\tau))}:=v_0(\tau,R)>0.$$
From the obvious inequality
$$ c_2(n,\tau,R) \theta(\varepsilon,\delta) \leq \frac{1}{v_0(\tau,R)^2} 
\int_{A_1 \times A_2} c_2(n,\tau,R) \theta(\varepsilon,\delta)\ \dvg(z_1)\dvg(z_2).$$
We get 
\begin{equation}
\int_{A_1 \times A_2} \frac{|\ell(F_c \circ 
\gamma_{z_1z_2}) - \ell(\gamma_{z_1z_2})|^2 }{d(z_1,z_2)} \leq 
\int_{A_1 \times A_2}  \frac{c_2(n,\tau,R)\theta(\varepsilon,\delta)}{v_0(\tau,R)^2}.\label{e2.17}
\end{equation}
As a consequence there exist $z_1 \in A_1$ and $z_2 \in A_2$ such that
$$ |\ell(F_c \circ 
\gamma_{z_1z_2}) - \ell(\gamma_{z_1z_2})|^2 \leq d(z_1,z_2)
   \frac{c_2(n,\tau,R)\theta(\varepsilon,\delta)}{v_0(\tau,R)^2}\,. $$
On the other hand one can check that by definition of $\tau$,
  $$\frac{c_2(n,\tau,R)\theta(\varepsilon,\delta)}{v_0(\tau,R)^2} = 
  \theta(\varepsilon,\delta) \frac{2 c_1(n,2R) \vol_{\HH^n}(1+R+1)^2}{v_n^2 \vol_{\HH^n}(\tau)} \tau = \tau^2.$$
  
This yields 
$$ |\ell(F_c \circ 
\gamma_{z_1z_2}) - \ell(\gamma_{z_1z_2})|^2 \leq d(z_1,z_2) \tau^2,$$
and 
$$ d(F_c(z_1),F_c(z_2)) \leq \ell(F_c \circ 
\gamma_{z_1z_2}) \leq d(z_1, z_2) + \tau \sqrt{d(z_1,z_2)}.
$$
Since $d(y_i,z_i) < \tau$ and $d(y_1,y_2)\geq \sqrt\tau$, we have 
$$ d(z_1,z_2) \leq d(y_1,y_2) + 2 \tau \leq d(y_1,y_2)(1+2\sqrt{\tau}).$$
With our choice of $\tau$ very small compared to $1$, we also have 
$$ d(z_1,z_2) \geq d(y_1,y_2) - 2 \tau \geq \frac{\sqrt{\tau}}{2}.$$
 
We then have
\begin{eqnarray}
d(F_c(y_1),F_c(y_2)) & \leq & d(F_c(y_1),F_c(z_1)) + 
d(F_c(z_1),F_c(z_2))+ d(F_c(z_2),F_c(y_2))\\
 & \leq & 2\sqrt{n}\tau + d(z_1,z_2) + \tau (d(z_1,z_2))^{1/2}
+2\sqrt{n}\tau \\
 & \leq & 4\sqrt{n}\tau + d(y_1,y_2)\frac{d(z_1,z_2)}{d(y_1,y_2)}(1+ \tau
(d(z_1,z_2))^{-1/2})\\
& \leq & 4\sqrt{n}\tau + d(y_1,y_2) (1+2\sqrt{\tau})
(1+\sqrt{2}\tau^{3/4} )\\
& \leq & 4\sqrt{n}\tau + d(y_1,y_2) (1+8\sqrt{\tau})\,.
\end{eqnarray}
We finally get
\begin{equation}
d(F_c(y_1),F_c(y_2)) \leq \kappa + (1+\kappa) d(y_1,y_2)\,, \label{e2.18}
\end{equation}
in case ii). 
\end{proof}





\section{A limit map on the limit space}

In this section, we consider a sequence $(Y_{k},g_{k})_{k\in \NN}$ of closed Riemannian $n$-manifolds satisfying the curvature bound (\ref{e0.1}) and the following assumption: we suppose that there exist an closed hyperbolic $n$-manifold $(X,g_{0})$, degree one maps 
$f_{k} : Y_{k} \rightarrow X$ and a sequence $\varepsilon_{k} \ra 0$  such that 
\begin{equation}
\vol_{g_{k}}(Y_{k}) \ra \vol_{g_{0}}(X)\,, \label{as_vol} 
\end{equation}
as $k$ goes to $+\infty$. From (\ref{e1.1}), for every $k \in\NN $, there exists $y_{g_{k}} \in Y_{k}$ satisfying the local volume estimate, that is $\vol (B_{g_k}(y_{g_k}, 1))\geq v_n>0$. For the sake of simplicity we shall use the notation $y_k$ instead of $y_{g_k}$.

Below, we prove that $(Y_{k},g_k,y_k)$ sub-converges in the pointed Gromov-Haudorff topology to a limit metric space $(Y_{\infty},d_{\infty},z_{\infty})$. Moreover, there exists a sequence of natural maps $F_{c_k} : (Y_k,g_k) \rightarrow (X,g_0)$, with suitably chosed parameters $c_k$,  which sub-converges to a "natural map" $ F : Y_{\infty} \longrightarrow X$. 

Let us  recall the definition of  the Gromov-Hausdorff topology. For two subsets $A,B$ of a metric space $Z$ the Hausdorff distance between $A$ and $B$ is 
$$\dh^Z(A,B) := \inf \{\varepsilon>0 \mid B \subset V_\varepsilon(A) \text{ and } 
A \subset V_\varepsilon(B)\} \in \RR \cup\{\infty\}.$$ 
It is a distance on compact subsets of $Z$ (see \cite{Fed}).

\begin{definition}[\cite{Gro}] Let $X_1$, $X_2$ be two metric spaces, then the Gromov-Hausdorff distance $\dgh(X_1,X_2) \in \RR \cup {\infty}$ is the infimum of the numbers 
 $$\dh^Z(f_1(X_1),f_2(X_2)))$$
 for all metric spaces $Z$ and all isometric embeddings $f_i : X_i \rightarrow Z$.
\end{definition}
It is a distance on the space of isometry classes of compact metric spaces. One says that a sequence $(X_i)_{i \in \NN}$ of metric spaces converges in the Gromov-Hausdorff topology to a metric space $X_\infty$ if $\dgh(X_i,X_\infty) \rightarrow 0$ as $i \rightarrow \infty$. Let $x_i \in X_i$ and $x_\infty \in X_\infty$, one says that the 
sequence $(X_i,x_i)_{i \in \NN}$ \emph{converges to $(X_\infty,x_\infty)$ in the pointed Gromov-Hausdorff topology} if for any $R>0$,  $\dgh(B_{X_i}(x_i,R), B_{X_\infty}(x_\infty,R)) \rightarrow 0$ as $i \rightarrow +\infty$ (in fact this definition holds only for length spaces, which will be sufficient in our 
situation). 

To deal with the Gromov-Hausdorff distance between $X_1$ and $X_2$, it is convenient to 
avoid the third space $Z$ by using $\varepsilon$-approximations between $X_1$ and 
$X_2$\,.

\begin{definition} Given two metric spaces $X_1$,$X_2$ and $\varepsilon>0$,  an $\varepsilon$-approximation (or $\varepsilon$-isometry) from $X_1$ to $X_2$ is a map 
$f : X_1 \rightarrow X_2$ such that 
\begin{enumerate}
\item for any $x,x' \in X_1$, $|d_{X_2}(f(x),f(x'))-d_{X_1}(x,x')| < \varepsilon$.
\item the $\varepsilon$-neighbourhood of $f(X_1)$ is equal to $X_2$.
\end{enumerate}
\end{definition}

Then one can show (see \cite[Corollary 7.3.28]{BBI}) that $\dgh(X_1,X_2) < \varepsilon$ 
if  there exists a $2\varepsilon$-approximation from $X_1$ to $X_2$ and similarly an $\varepsilon$-approximation exists if $\dgh(X_1,X_2) < 2\varepsilon$. Let us insist on the fact that these approximations may be neither continuous nor even measurable.  

Our goal is to prove the :

\begin{proposition} \label{prop3.3}
Up to extraction and renumbering, the sequence $(Y_{k},g_k,y_{k})$ satisfies the following. 
\begin{enumerate}
\item There exists a complete pointed length space $(Y_{\infty},d_{\infty},y_{\infty})$   such that $(Y_{k},g_k,y_{k})$ converges in the pointed Gromov-Hausdorff topology to a metric space $(Y_{\infty},d_{\infty},y_{\infty})$. Moreover, $(Y_{\infty},d_{\infty})$ has Hausdorff dimension equal to $n$.
\item there exist sequences of positive numbers $\varepsilon_k$, $\delta_{k}$ going to $0$,
  $c_k$ such that $h(g_{k}) < c_{k}< h(g_{k})+\delta_{k}$, $R_{k}$ going to $+\infty$ such that $\varepsilon_{k} \leq \varepsilon(R_{k})$ and $\delta_{k}\leq \delta(R_{k})$. There also exist  
 and $\alpha_{k}$-approximations $\psi_{k} : B_{d_{\infty}}(y_{\infty},R_{k}) \ra B_{g_{k}}(y_{g_{k}},R_{k})$ such that the following holds. 
 Let $$F_{c_{k}} : (Y_{k},g_{k}) \ra (X,g_{0})$$
 be the natural map as defined in section 2. Then  $F_{c_{k}}\circ \psi_{k}$ converges uniformly on compact sets to a map $$F : Y_{\infty} \longrightarrow X,$$
which is $1$-lipschitz. 
\end{enumerate}
\end{proposition}
The proof is divided in two steps described in the following sections.

\subsection*{Existence of the limit and its properties} 

Under the curvature bound (\ref{e0.1}) and the local volume estimate (\ref{loc-vol}), (1) 
of Proposition \ref{prop3.3}
is a straightforward application of Gromov \& Cheeger-Colding compactness theorem, see \cite[Theorem 1.6]{Ch-Co}.
Before proving point (2) of 
Proposition \ref{prop3.3}, let us describe some features of the convergence and of the limit space which
will be used later.  

The continuity  of the volume under the (pointed) Gromov-Hausdorff convergence is crucial for our purposes. For $\ell>0$, note $\mathcal{H}^\ell$ the $\ell$-dimensional Hausdorff measure of a metric space (see \cite{BBI} definition 
1.7.7). 
\begin{theorem}[\cite{Ch-Co}, Theorem 5.9]
\label{t3.1}
Let $p_i \in Y_i$  and $p_{\infty} \in Y_{\infty}$ their limit, and let $R>0$. 
Then 
\begin{equation} \lim_{i \ra +\infty} \vol_{g_i}(B(p_i,R)) = \H(B(p_\infty,R))\,.
\label{e3.2}
\end{equation}
\end{theorem}

In particular, $Y_\infty$ satisfies the Bishop-Gromov inequalities (\ref{BG}) and the Bishop inequality. 
By definition, a \emph{tangent cone at} $p \in Y_{\infty}$  is a complete pointed 
Gromov-Hausdorff limit, $\{Y_{\infty,p},d_\infty,p_\infty\}$ of a sequence of 
rescaled space, $\{(Y_{\infty},r_i^{-1}d,p)\}$, where $\{r_i\}$ is a positive sequence 
such that $r_i \ra 0$. Indeed, by \cite[Proposition 5.2]{GLP}, every such sequence has 
a convergent subsequence, but the limit might depend on  the choice of the 
sub-sequence. Notice that this notion is different from the one  described in \cite[Chapter 8]{BBI} where the authors require that the limit is unique (does not depend on the sub-sequence). 

\begin{definition} \label{singular}
The $\mathrm{regular}$ set $\cR$ consists of those points, $p \in Y_{\infty}$, 
such that every tangent cone at $p$ is isometric to $\RR^n$. 
The complementary $\cS = Y_{\infty}\setminus \cR$ is the $\mathrm{singular}$ set.  
\end{definition}

Let $B_0^n(1) \subset \RR^n$ be the unit ball. 

\begin{definition} The $\varepsilon$-regular 
set $\cRe$ consists of those points, $p \in Y_{\infty}$, such that every tangent 
cone, $(Y_{\infty,p},p_\infty)$, satisfies $d_{GH}(B(p_\infty,1),B_0^n(1))< \varepsilon$. 
A point in $Y_{\infty}\setminus\cRe=\cSe$ is called $\varepsilon$-singular, 

\end{definition}

\begin{theorem}[\cite{Ch-Co}, Theorem 5.14]
\label{t3.2} There exists $\varepsilon_{n}>0$ such that 
 for $\varepsilon \leq \varepsilon_n$, $\overset{\circ}{\cRe}$ has a natural smooth manifold structure. Moreover, for this parametrization, the metric on 
  $\overset{\circ}{\cRe}$  is bi-h\"older equivalent to a smooth Riemannian metric.
 The exponent $\alpha(\varepsilon)$ in this bi-h\"older equivalence satisfies $\alpha(\varepsilon) \ra 1$ as 
$\varepsilon \ra 0$. 
\end{theorem}

\begin{theorem}[\cite{Ch-Co}, Theorem 6.1]
\label{t3.3}
\begin{equation}
 \mathcal{H}^{n-2}(\cS) = 0 \label{e3.3}
\end{equation}
\end{theorem}

\begin{remark}
\label{r3.1}
Clearly, $\cR = \cap_{\varepsilon>0} \cRe$. The sets $\cRe$, $\cR$ are 
not necessarily open. However, for any $\varepsilon >0$, there 
is some $\varepsilon>\delta>0$ such that $\cR_\delta \subset \overset{\circ}{\cRe}$ 
(see \cite[Appendix A.1.5]{Ch-Co}). 
In  \cite[Section 3]{Ch-Co2}, it is also proved that $\overset{\circ}{\cRe}
$ is path connected. This important fact will be used in the last part 
of this text.
\end{remark}

We now study the density of the Hausdorff measure. A consequence of  Bishop's  inequality is that 
 $$\limsup_{r\ra0} \frac{\H(B(p,r))}{\vol_{\RR^n}(r)}\leq 1.$$ 
\begin{definition} The density at $p$ of $Y_\infty$ is 
\begin{equation}
\theta(p):=\liminf_{r\ra 0} \frac{\H(B(p,r))}{\vol_{\RR^n}(r)}. \label{density}
\end{equation}
\end{definition}
A consequence of \cite[A.1.5]{Ch-Co} is the existence of some 
positive function 
$\tau(\varepsilon)$, with $\tau(\varepsilon) \ra 0$ as $\varepsilon \ra 0$, such that 
for every $p \in \cRe$, 
\begin{equation}
 \theta(p) > 1-\tau(\varepsilon).
\end{equation}

Conversely, there exists a positive function $\varepsilon(\tau)$, satisfying
 $\varepsilon(\tau) \ra 0$ as $\tau \ra 0$ and such that 
\begin{equation}
  \theta(p) \geq 1-\tau  \Longrightarrow p \in \cR_{\varepsilon(\tau)}\,.
\end{equation}

\begin{remark}\label{rk-density}
A point $p$ is regular if and only if $\theta(p) = 1$. From now on, we consider $\varepsilon \leq \varepsilon_0$, where $\varepsilon_0 \leq 
\varepsilon_n$ is sufficiently small so that $\tau(\varepsilon_0)
< 1/2$, the density is thus strictly greater than $1/2$ on $\cRe$. 
\end{remark}


\subsection*{Existence of the natural map at the limit} 
Let us now prove (2) of Proposition \ref{prop3.3}.

\begin{proof}

For every $k \in \NN$ and $c>h(g_k)$, there exists a natural map 
$F_c : (Y_k,g_k) \rightarrow (X,g_0)$, described in Section 2. We need to choose the values of $c$ for each $g_k$ in order that $F_c$ to satisfies some good properties. One argues as follows. 

Given $m \in \NN^*$, one chooses positive numbers $\varepsilon_m \leq \varepsilon_{2}(m)$ and 
$\delta_m \leq \delta_{2}(m)$ sufficiently small such that $\kappa(\varepsilon_m,\delta_m,m) \leq \frac{1}{m}$, where $\delta_2$, $\varepsilon_2$ and $\kappa$ are given by Lemma \ref{l2.6}. One then defines
$$\alpha_m = \max\Big\{\alpha_1(\varepsilon_m,\delta_m),\alpha_2(\varepsilon_m,\delta_m), \alpha_3(\varepsilon_m,\delta_m)
\kappa(\varepsilon_m,\delta_m,m)\Big\}.$$
We check that  $\alpha_m \ra 0$ as $m \ra +\infty$. By the hypothesis (\ref{as_vol}), there exists 
$k_1(m) \in \NN$ such that for any $k \geq k_1(m)$, $\vol_{g_k}(Y_k) \leq (1+\varepsilon_m)
\vol_{g_0}(X)$. Since for $m$ fixed $B_{g_k}(y_k,m)$ converges to $B_{\infty}(y_\infty,m)$, there exists $k_2(m) \in \NN$ such that 
for any $k \geq k_2(m)$, there exists $\alpha_m$-approximations from $B_\infty(y_\infty,m)$ to $B_{g_k}(y_k,m)$. Define $k(m):= \max\{k_1(m),k_2(m)\}$ and let 
 $\psi_m:B_\infty(y_\infty,m)\lra B_{g_k}(y_{k(m)},m)$ be an $\alpha_m$-approximation. One can 
 assume that $\psi_m(y_\infty) = y_{g_{k(m)}}$. Choose $h(g_k)<c_m<h(g_k) + \delta_m$ and 
 consider  $$ F_{c_m} \circ \psi_m : B_\infty(y_\infty,m)\lra X.$$
 Lemma \ref{l2.6} applies to $F_{c_m}$ on $B_{g_{k(m)}}(y_{k(m)},m)$. Hence, for any 
 $p,q \in B_\infty(y_\infty,m)$, 
 \begin{eqnarray*}
 d_{g_0}(F_{c_m} \circ \psi_m(p),F_{c_m} \circ \psi_m(q))& \leq & 
(1+\alpha_m) d_{g_k}(\psi_m(p),\psi_m(q)) + \alpha_m\\
& \leq & (1+\alpha_m)d_\infty(p,q) + (1+\alpha_{m})\alpha_m + \alpha_m.
\end{eqnarray*}
Applying the same reasoning as in Ascoli's theorem, one can show that for any compact $K \subset Y_\infty$, there exists a  sub-sequence of $F_{c_m}$ converging to a map $F_K : K \rightarrow X$. We denote it by $F_{c_{\phi (m)}}$. If one uses an exhaustion of $Y_\infty$ by compact sets and a standard diagonal process, one can extract a sub-sequence of $F_{c_{\phi (m)}}\circ \psi_{\phi (m)}$ which converges uniformly on any compact set to a map $F:Y_\infty \ra X$. It is easy to see that  the map $F$ is 1-lipschitz.

Then one renumbers the sub-sequences $Y_{k(\phi(m))}$, $\psi_{\phi(m)}$ and $F_{c_{\phi(m)}}$ 
such that, for any $m \in \NN^*$, 
$\vol_{g_m}(Y_m) \leq (1+\varepsilon_m) \vol_{g_0}(X)$, $h(g_m) < c_m < c_m + \delta_m$, 
 the inequalities of Lemmas \ref{l2.2}, \ref{l2.4} hold with $\alpha_1$, $\alpha_2$,  $\alpha_3$
replaced by $\alpha_m$ and those of Lemmas \ref{l2.5}, \ref{l2.6} hold on $B(y_m,m)\subset Y_m$ with $\kappa$ replaced by $\alpha_m$. For simplicity, the map $F_{c_m}$ will be denoted $F_m$.
\end{proof}


\section{The limit map $F:Y_{\infty} \lra X$ is  isometric}


In this section we aim at proving that the limit map $F=\lim F_{k} \circ \psi_k $ is an isometry, 
\textsl{i.e.} it is distance preserving.
We prove first that $F$ preserves the volume.
\begin{lemma}
\label{l4.1}
Let $A \subset Y_{\infty}$ be a measurable subset. Then,
\begin{equation}
\vol_{g_0}(F(A)) = \H(A)\,.  \label{e4.1}
\end{equation}
\end{lemma}

\begin{proof} 
It suffices to prove the lemma when the set $A$ is an open ball. Indeed, let us assume that $F$ preserves the volume of balls and let $A$ be a measurable set included in a ball $B:=B_\infty (p, r)$. Since $F$ is contracting it does not increase the volumes (see \cite[Proposition 3.5]{Mor}). Now, if $\vol_{g_0}(A)<  \H(A)$ and since we have $\vol_{g_0}(B\setminus A)\leq  \H(B\setminus A)$ we have a contradiction with the preservation of the volume of $B$. Similarly, if $A$ is a measurable set of finite measure we can apply the same argument with $A$ and $B\setminus A$ for any ball $B$.

It is then enough  to prove 
that for every $B_\infty(p,r) \subset Y_{\infty}$, $\vol_{g_0}(F(B_\infty(p,r)))\geq \H(B_\infty(p,r))$. 
By construction, $\overline{F(B_\infty(p,r))}$ is  the Hausdorff 
limit of $\overline{F_k \circ \psi_k(B_\infty(p,r))}$. 

We first  show that  this is also the Hausdorff limit of $\overline{F_k(B_{g_k}(\psi_k(p),r))}$. 
Let $x \in \overline{F(B_\infty(p,r))}$ and $x_{k} \in F(B_\infty(p,r))$ such that $x_{k}\ra x$. Let $p_{k} \in B_\infty(p,r)$ such that 
$F(p_{k}) = x_{k}$. By definition of the $\alpha_{k}$-approximation, one has 
$d_{g_{k}}(\psi_{k}(p_{k}),\psi_{k}(p)) < r + \alpha_{k}$. There exists  $z_k \in B_{g_k}(\psi_{k}(p),r)$ such that 
$d_{g_{k}}(\psi_{k}(p_{k}),z_{k}) < \alpha_{k} $ (for example $z_{k}$ may be on the segment $[\psi_{k}(p_{k}),\psi_{k}(p)]$). 
Note that, by the triangular inequality,  $d_\infty(p_{k},y_{\infty}) \leq r + d_\infty(p,y_{\infty})$ and recall that $\psi_{k}(y_\infty) = y_{g_{k}}$. Thus 
$\psi_{k}(p_{k})$ remains at bounded distance from $y_{g_{k}}$. Then, 
applying Lemma \ref{l2.6} we have    
\begin{eqnarray*}
 d_{g_{0}}(F_{k}(z_k),F_{k}(\psi_{k}(p_{k})) & \leq & (1+\alpha_{k}) d_{g_{k}}(z_k,\psi_{k}(p_{k})) + \alpha_{k}\\
 		& \leq & (1+\alpha_{k}) \alpha_{k}+ \alpha_{k}\\
		& \underset{k\to +\infty}{\longrightarrow} & 0.
\end{eqnarray*}	
On the other hand, since $F_{k}\circ \psi_k$ converges uniformly to $F$ on compact sets, $ F_{k}(\psi_{k}(p_{k}))$ has the 
same limit as $F(p_{k})=x_{k}$, that is $ F_{k}(\psi_{k}(p_{k})) \ra x$. From the inequality above one deduces that 
$F_{k}(z_k) \ra x$ which shows that $x \in \lim_{k \ra \infty} 	\overline{F_k(B_{g_k}(\psi_k(p),r))}$. 
One has then proved that $\overline{F(B_\infty(p,r))} \subset \lim_{k \ra \infty} 	\overline{F_k(B_{g_k}(\psi_k(p),r))}$. In order to prove the other inclusion one argues similarly. Given $x \in \lim_{k \ra \infty} 	\overline{F_k(B_{g_k}(\psi_k(p),r))}$, there exists 
$x_k \in F_k(B_{g_k}(\psi_k(p),r))$ such that $x_k \ra x$, with 
$x_k = F_k(z_k)$ where $z_k \in B_{g_k}(\psi_k(p),r)$. As $\psi_k$ is an $\alpha_k$-approximation from $B_\infty(y_\infty,k)$ to $B(y_{g_k},k)$, one has the inclusion
$B_{g_k}(\psi_k(p),r)\subset U_{\alpha_k} \psi_k(B_\infty(p,r+\alpha_k))$ for large $k$, thus there exists 
$q_k \in B_\infty(p,r+\alpha_k)$ satisfying  $d_{g_k}(z_k,\psi_k(q_k))<\alpha_k$. As 
$Y_\infty$ is a length space, there exists $q'_k \in B_\infty(p,r)$ such that 
$d_\infty(q_k',q_k) < \alpha_k$. Then 
$d_{g_k}(\psi_k(q_k'),z_k) \leq d_{g_k}(\psi_k(q_k'),\psi_k(q_k)) + 
d_{g_k}(\psi_k(q_k),z_k)) < 3\alpha_k$. 
Thus 
\begin{eqnarray*}
d_{g_0}(F_k \circ \psi_k(q_k'),x_k)= d_{g_0}(F_k \circ \psi_k(q_k'),F_k(z_k)) 
& \leq & (1+\alpha_k) d_{g_k}(\psi_k(q_k'),z_k) + \alpha_k \\
& \leq & (1+\alpha_k)3\alpha_k + \alpha_k \ra 0.
\end{eqnarray*}
 Hence $d_{g_0}(F_k \circ \psi_k(q'_k),x) \ra 0$. 
As $ F_k \circ \psi_k$ converges uniformly to $F$ on compact sets, 
one has $d_{g_0}(F(q'_k),x) \ra 0$ thus $x \in \overline{F(B_\infty(p,r))}$.  This shows that $x \in \overline{F(B_\infty(p,r))}$ is the Hausdorff limit of $\overline{F_k(B_{g_k}(\psi_k(p),r))}$.

In order to prove the lemma it is then sufficient to prove that 
\begin{equation}\label{76}
\liminf_{k \ra +\infty}\vol_{g_0}(\overline{F_k(B_{g_k}(\psi_k(p),r)})\geq \liminf_{k \ra +\infty}\vol_{g_0}(F_k(B_{g_k}(\psi_k(p),r)) \geq \H(B_\infty(p,r))\,.
\end{equation}
Indeed, inequality (\ref{76}) will imply that 
$$\vol_{g_0}(F(\overline{B_\infty(p,r)}))\geq\vol_{g_0}(\overline{F(B_\infty(p,r))})\geq \H(B_\infty(p,r))$$ 
and thus
$\vol_{g_0}(F(B_\infty(p,r))) \geq \H(B_\infty(p,r))$
since $F$ being Lipschitz, we have 
$$\vol_{g_0}(F(\overline{B_\infty(p,r)}))=\vol_{g_0}(F(B_\infty(p,r))).$$ 
Recall that $N(F_k, x)$ is the number of preimages of $x$ by $F_k$. 
We denote by $X_{k,1}$ the set of $x\in X$ such that $N(F_k, x)=1$. 
The construction of the sequence $(F_k)$, Lemma \ref{l2.7} and our choice of the 
$\alpha_k$'s imply 
that $\vol_{g_0}(X_{k,1}) \geq (1-\alpha_k)\vol_{g_0}(X)$ and 
\begin{equation}
\int_{X\setminus X_{k,1}} N(F_k, x)\dvgo(x) \leq \alpha_k \vol_{g_0}(X)\,. \label{e4.2}
\end{equation}

We also denote by $Y_{k,\alpha_k}$ the set of $y \in Y_k$ such that 
\begin{equation}
1-\alpha_k \leq |\jac F_k(y)| \leq  1+\alpha_k.  \label{e4.3}
\end{equation}
Then Lemma \ref{l2.3} implies that 
$\vol_{g_k}(Y_{k,\alpha_k}) \geq (1-\alpha_k)\vol_{g_k}(Y_k)$, for $k$ large enough. We then have 
\begin{eqnarray}
\vol_{g_0}(F_k(B_{g_k}(\psi_k(p),r)))&= &  \int_{F_k(B_{g_k}(\psi_k(p),r))} \dvgo \nonumber\\
& = & \int_{F_k(B_{g_k}(\psi_k(p),r))\cap X_{k,1}} N(F_k, x)\dvgo(x) +
 \vol_{g_0}(F_k(B_{g_k}(\psi_k(p),r))\setminus X_{k,1}) \nonumber \\
&  \geq & \int_{B_{g_k}(\psi_k(p),r) \cap F_k^{-1}(X_{k,1})\cap Y_{k,\alpha_k}} |\jac F_k(y)| \textrm{dv}_{g_k}(y)  \nonumber \\
& \geq & (1-\alpha_k)\vol_{g_k}\left(B_{g_k}(\psi_k(p),r) \cap F_k^{-1}(X_{k,1}) \cap Y_{k,\alpha_k}\right). \label{e4.4}
\end{eqnarray}

On the other hand, using (\ref{e4.3}) and (\ref{e4.2}) we have
\begin{eqnarray*}
 \vol(F_k^{-1}(X\setminus X_{k,1})\cap Y_{k,\alpha_k})
& \leq & \int_{F_k^{-1}(X\setminus X_{k,1})\cap Y_{k,\alpha_k}} \frac{|\jac F_k|}{1-\alpha_k} \textrm{dv}_{g_k} \\ 
& \leq & \frac{1}{1-\alpha_k}  \int_{X\setminus X_{k,1}} N(F_k, x)\dvgo(x) \\ 
& \leq & \frac{\alpha_k}{1-\alpha_k} \vol_{g_0}(X),
\end{eqnarray*}
consequently   
\begin{align*}
\begin{split}
\vol_{g_k}(B_{g_k}(\psi_k(p),r) \cap F_k^{-1}(X_{k,1}) \cap Y_{k,\alpha_k})& = 
\vol_{g_k}(B_{g_k}(\psi_k(p),r) \cap Y_{k,\alpha_k}) \\ 
& - \vol_{g_k}(B_{g_k}(\psi_k(p),r)\cap F_k^{-1}(X\setminus X_{k,1})\cap Y_{k,\alpha_k})
\end{split}\\
& \geq  \vol_{g_k}(B_{g_k}(\psi_k(p),r)) - \alpha_k \vol_{g_k}(Y_k) - \frac{\alpha_k}{1-\alpha_k}\vol_{g_0}(X).
\end{align*}
Plugging this inequality in (\ref{e4.4}) one gets 
$$  \vol_{g_0}(F_k(B_{g_k}(\psi_k(p),r))  \geq  
 (1-\alpha_k) \vol_{g_k}(B_{g_k}(\psi_k(p),r)) - (1-\alpha_k)  
 \alpha_k \vol_{g_k}(Y_k) - \alpha_k \vol_{g_0}(X). 
$$
As $B_{g_k}(\psi_k(p),r)$  converges to $B_\infty(p,r)$ in the Gromov-Hausdorff topology, 
Theorem \ref{t3.1} implies that $\lim_{k\ra \infty} \vol_{g_k}(B_{g_k}(\psi_k(p),r)) = 
\H(B_\infty(p,r))$, hence 
$$ \liminf_{k\ra \infty} \vol_{g_0}(F_k(B_{g_k}(\psi_k(p),r)) \geq \H(B_\infty(p,r)),$$
which proves the lemma. 
\end{proof}

We now prove that $F$ is injective on the set of points where the density 
is larger than $1/2$.
\begin{lemma}
\label{l4.2}
The map $F$ is injective on $\cRe$ for $\epsilon\leq \epsilon_0$.  
\end{lemma}
\begin{proof}
Suppose that there are $p_1$,$p_2 \in \cRe$ such that $F(p_1) = F(p_2)$. As $F$ is 
1-lipschitz, we have for every $r>0$,
$$ F\left(B_\infty(p_1,r) \cup  B_\infty(p_2,r)\right) \subset B_{g_0}(F(p_1),r)\,.$$
By the previous lemma,
\begin{eqnarray} 
\H \left(B_\infty(p_1,r) \cup  B_\infty(p_2,r)\right) & = &
 \vol_{g_0}\left(F(B_\infty(p_1,r) \cup  B_\infty(p_2,r))\right)\nonumber  \\
  & \leq & \vol_{g_0}\left(B_\infty(F(p_1),r)\right). \label{e4.8}
\end{eqnarray}
For $r< d(p_1,p_2)/2$ the balls $B_\infty(p_1,r)$ and $B_\infty(p_2,r)$ are disjoint. Hence, dividing (\ref{e4.8}) by $\vol_{\RR^n}(r)$, we get
$$ \frac{\H(B_\infty(p_1,r))}{\vol_{\RR^n}(r)} +\frac{\H(B_\infty(p_2,r))}{\vol_{\RR^n}(r)}
 \leq \frac{\vol_{g_0}\left(B_{g_0}(F(p_1),r)\right)}{\vol_{\RR^n}(r)}.$$
Taking the liminf as $r \ra 0$ yields 
$$ \theta(p_1) + \theta(p_2) \leq \theta(F(p_1))=1,$$
which is a contradiction, since $\theta >1/2$ on $\cRe$ if $\varepsilon<\varepsilon_0$ 
(see remark \ref{rk-density}). 
\end{proof}


\begin{lemma}
\label{l4.3}
 The map $F$ is open on $\overset{\circ}{\cRe}$ for $\epsilon\leq \epsilon_0$.
\end{lemma}
\begin{proof} 
Let $p \in \overset{\circ}{\cRe}$. We have to prove that 
there exists $\eta>0$ such that $B_{g_0}(F(p),\eta) \subset F(\overset{\circ}{\cRe})$. 
There exists $r>0$ such that $B_\infty(p,2r) \subset \overset{\circ}{\cRe}$. 
For the sake of simplicity we shall note $B:= B_\infty(p,r)$. By the previous lemma, $F(p) \notin F(\partial B)$. Thus, by compactness of $\partial B$ and continuity of $F$, there exists $\eta >0$ such that 
$d_{g_0}(F(p), F(\partial B)) > \eta$. Notice that, since $F$ is $1$-Lipschitz, $\eta<r$. Here, one could use the theory of 
local degree as in \cite[Appendix C]{BCG}, however  $Y_\infty$ is not, \textsl{a priori} a manifold and it may even be not  locally lipschitz equivalent to $\RR^n$. Let $R>2r+d_\infty (y_\infty , p)$ be a fixed radius; it satisfies  
$\psi_k(B_\infty(p,2r)) \subset B_{g_k}(y_{g_k},R)$ for large $k$. 
Let $z_k = \psi_k(p)$ and $B_k :=B(z_k,r)$. The choice of $R$ and the fact that the $\psi_k$'s are approximations 
shows that  $B_k\subset B(y_{g_k},R)$, for $k$ large enough. 
We choose  $k$ large enough such that
$\dh(F_k(\partial B_k),F(\partial B)) \leq \frac{\eta}{10}$. 
This is possible since $\dh (\psi_k(\partial B), \partial B_k)$ goes to zero, 
$F_k \circ \psi_k$ converges to $F$ and $F(p)$ is at distance from $F(\partial B)$ larger than $\eta$.
Let $\cC$ (resp. $\cCk$) be  the connected component of $X\setminus F(\partial B)$ 
(resp. $X\setminus F_k(\partial B_k)$), which contains $F(p)$, 
(resp $F_k(z_k)$). Now the ball $B(F(p),\eta /10)$ is included in $\cC$ and for $k$ large enough
$B(F_k (z_k),\eta /10)$ is included in $\cC _k$. On the other hand
by Corollary 4.1.26 of \cite{Fed}, $\degree(F_k|B_k)$ is constant on $\cCk$, 
where, for a subset $A \subset Y_k$, $$ \degree(F_k|A)(x)= \sum_{y \in F_k^{-1}(x)\cap A} \textrm{sign }\jac F_k(y).$$
We show that  $\degree(F_k|B_k)=1$ on $\cCk$ as follows. We have to show that at least one point 
in $\cCk$ this degree is $1$ since it is constant on this set. In order to do that, we shall show that the set 
of such points has positive measure. Denote again by 
$X_{k,1} \subset X$ the set of $x \in X$ such that $N(F_k, x)=1$, 
that is $x$ has one preimage by $F_k$. By Lemma \ref{l2.7}, 
$\vol_{g_0}(X_{k,1}) \geq (1-\alpha_k)\vol_{g_0}(X)$. The intersection of $X_{k,1}$ with $\cCk$ has a positive measure for $k$ large enough; indeed,  
$B(F_k(z_k),\frac{\eta}{10})\subset \cCk$ and its volume is bounded below by (\ref{loc-vol}) and 
$\vol( B(F_k(z_k),\frac{\eta}{10})\setminus X_{k,1} ) \lra 0$ as $k \ra +\infty$.  Now, by Lemma \ref{l2.5}  
one has $F_k(B(z_k,\frac{\eta}{20\sqrt{n}})) \subset B(F_k(z_k),\frac{\eta}{10})$  and $B(z_k,\frac{\eta}{20\sqrt{n}})\subset B_k$
for large $k$, and  an argument similar to the one used in \ref{76} shows that 
the volume of the image is bounded below. It thus intersects $X_{k,1}$ on a 
set of positive measure for $k$ large enough. This proves that 
$\degree(F_k|B_k)=1$ on $\cCk$. Since $B(F_k(z_k), \eta/10)$ converges to $B(F(p), \eta/10)$, this last ball is included in $\cCk$ for $k$ large; hence, any point 
in $B(F(p),\frac{\eta}{10})$ has a preimage by $F_k$ in $B_k$. 
By taking the limit when $k$ goes to $+\infty$, we get $B(F(p),\frac{\eta}{10}) 
\subset F(\overline{B(p,r)}) \subset F(B(p,2r)) \subset F(\overset{\circ}{\cRe})$. 
\end{proof}  


\begin{lemma}
\label{l4.4}
There exists $c(\varepsilon) >0$ such that 
$F:\overset{\circ}{\cRe} \lra F(\overset{\circ}{\cRe} )\subset X$ is locally 
$(1+c(\varepsilon))$-bi-Lipschitz. Moreover, $c(\varepsilon) \ra 0$ as $\varepsilon \ra 0$.
\end{lemma}

\begin{proof} 
The idea is the following: we already know that $F$ is $1$-lipschitz and volume preserving. In particular, a ball $B_\infty(p,r) \subset Y_\infty$ is sent into a ball $B_{g_0}(F(p),r) \subset X$. If the ball in $Y_\infty$ is in the almost regular part 
and has a small radius, its volume is close to the Euclidean one, so is 
the volume of the hyperbolic ball. One can then estimate how much the image 
of $B_\infty(p,r)$ is close to fill $B_{g_0}(F(p),r)$. If one considers the images of two disjoint balls, one can estimate how the corresponding hyperbolic balls overlapp, and thus the distance between their centers. 

Let $p \in \overset{\circ}{\cRe} $. Let $r(p,\varepsilon)>0$ be a radius such 
that for every $0< r \leq r(p,\varepsilon)$,
$$\frac{\H(B_\infty(p,r))}{\vol_{\RR^n}(r)} \geq 1-\tau(\varepsilon),$$
and let $r_\varepsilon = \min \{ \varepsilon ,r(p,\varepsilon )\}$. One can assume 
that $r_\varepsilon$ is smaller than the injectivity radius of $X$. 
Let $0<r< r_\varepsilon^2$ be such that $B_\infty(p,r) \subset \cRe$.  For every 
$q \in B_\infty(p,r)$, $B_\infty(p,r_\varepsilon-r_\varepsilon^2) \subset 
B_\infty(q,r_\varepsilon)$. Thus, 
\begin{eqnarray}
\H(B_\infty(q,r_\varepsilon))& \geq & \H(B_\infty(p,r_\varepsilon-r_\varepsilon^2))\\
 &\geq & (1-\tau(\varepsilon))\vol_{\RR^n}(r_\varepsilon-r_\varepsilon^2)\\
& \geq & (1-\tau(\varepsilon))(1-r_\varepsilon)^n \vol_{\RR^n}(r_\varepsilon)\,.\label{e4.5}
\end{eqnarray}    
Suppose that there exists $p_1$,$p_2 \in B_\infty(p,r)$, $p_1 \not= p_2$ and a number $0<\rho<1$ such that $$ d_{g_0}(F(p_1),F(p_2)) \leq \rho d_\infty (p_1,p_2).$$
Define $r'= d_\infty(p_1,p_2)/2>0$ and notice that $r'<r$. 
By (\ref{e3.2}) and the Bishop-Gromov inequality 
(\ref{BG}), for $i=1$, $2$ one has 
$$ \H(B_\infty(p_i,r')) \geq \H(B_\infty(p_i,r_\varepsilon))\frac{\vol_{\HH^n}(r')}{\vol_{\HH^n}(r_\varepsilon)}.$$
Thus, by Lemma \ref{l4.1}, (\ref{e4.5}) and Bishop-Gromov inequality we have
\begin{eqnarray}
\vol_{g_0}\left(F(B_\infty(p_1,r')\cup B_\infty(p_2,r'))\right) & = & \H(B_\infty(p_1,r')) +
\H(B_\infty(p_2,r'))\\
& \geq & 2(1-\tau(\varepsilon))(1-r_\varepsilon)^n \frac{\vol_{\HH^n}(r')}{\vol_{\HH^n}(r_\varepsilon)}
\vol_{\RR^n}(r_\varepsilon)\\
& \geq & 2(1-\tau(\varepsilon))(1-r_\varepsilon)^n 
 \frac{\vol_{\RR^n}(\varepsilon)}{\vol_{\HH^n}(\varepsilon)}\vol_{\RR^n}(r')\\
 & \geq & 2 \vartheta(\varepsilon) \vol_{\RR^n}(r')
  \label{e4.6}
\end{eqnarray}
where $\vartheta(\varepsilon)= (1-\tau(\varepsilon))(1-\varepsilon)^n 
 \frac{\vol_{\RR^n}(\varepsilon)}{\vol_{\HH^n}(\varepsilon)} 
 \ra 1$ as $\varepsilon \ra 0$. 	

On the other hand,
$$F(B_\infty(p_1,r')\cup B_\infty(p_2,r')) \subset B_{g_0}(F(p_1),r') \cup 
 B_{g_0}(F(p_2),r'),$$
 Hence 
 \begin{align}
 \begin{split}
 \vol_{g_0}\left(F(B_\infty(p_1,r')\cup B_\infty(p_2,r'))\right)  \leq \vol_{g_0}&(B_{g_0}(F(p_1),r')) 
 + \vol_{g_0}(B_{g_0}(F(p_2),r'))\\
 & - \vol_{g_0}(B_{g_0}(F(p_1),r') \cap B_{g_0}(F(p_2),r')).\label{vol-maj}
 \end{split} 
 \end{align}
For any $x \in X$ and any $s>0$ smaller than the injectivity radius of $X$ one has 
$\vol_{g_0}(B(x,s)) = \vol_{\HH^n}(s)$. Let $x$ be the middle point of the segment 
$[F(p_1)F(p_2)]$. Then  
$$B(x,r'(1-\rho)) \subset B(F(p_1),r') \cap B(F(p_2),r').$$
 Indeed, if $x' \in B(x,r'(1-\rho))$ then 
$d(x',F(p_i)) \leq d(x',x) + d(x,F(p_i)) < r'(1-\rho) + \rho r'=r'$ for $i=1$, $2$. 
Thus (\ref{vol-maj}) gives 
\begin{eqnarray}
\vol_{g_0}\left(F(B(p_1,r')\cup B(p_2,r')) \right)& \leq & 2\vol_{\HH^n}(r')-\vol_{\HH^n}(r'(1-\rho))\\
&\leq & 2 \vol_{\RR^n}(r')\frac{\vol_{\HH^n}(r')}{\vol_{\RR^n}(r')  }-(1-\rho)^n\vol_{\RR^n}(r')\\
& \leq & 2 \vol_{\RR^n}(r')\frac{\vol_{\HH^n}(\varepsilon)}{\vol_{\RR^n}(\varepsilon)  }-(1-\rho)^n\vol_{\RR^n}(r')\\
& = & \left(2 \frac{\vol_{\HH^n}(\varepsilon)}{\vol_{\RR^n}(\varepsilon)} - (1-\rho)^n\right) \vol_{\RR^n}(r').
\label{e4.7} 
\end{eqnarray}
For the third inequality we have used Bishop-Gromov's inequality.
From (\ref{e4.6}) and (\ref{e4.7}), we find
$$ (1-\rho)^n \leq 2\left(\frac{\vol_{\HH^n}(\varepsilon)}{\vol_{\RR^n}(\varepsilon)} 
-\vartheta(\varepsilon)\right) \ra 0,$$ 
therefore
$$ \rho \geq 1 - 2^{1/n}\left(\frac{\vol_{\HH^n}(\varepsilon)}{\vol_{\RR^n}(\varepsilon)} 
-\vartheta(\varepsilon)\right)^{1/n} := 1 - c(\varepsilon) \ra 1,$$ 
as $\varepsilon \ra 0$. One has proved that inside the ball $B(p,r)$, 
$$d_{g_0}(F(p_1),F(p_2)) \geq (1-c_1(\varepsilon)) d_\infty(p_1,p_2),$$
and the proof of the lemma follows by choosing $c(\varepsilon )$ so that $1-c_1(\varepsilon)\geq (1+c(\varepsilon))^{-1}$. 
\end{proof}

\begin{remark} On the connected (see Remark \ref{r3.1}) open set $F(\overset{\circ}{\cRe}) \subset X$, the metric $g_0$ induces a distance $\rho_\varepsilon$. 
The above lemma shows that $F:(\overset{\circ}{\cRe},d_\infty) \lra (F(\overset{\circ}{\cRe} ),\rho_\varepsilon)$ is a $(1+c(\varepsilon))$-bi-Lipschitz homeomorphism. If one can 
prove that $\rho_\varepsilon=d_{g_0}$, one deduces that $\cRe$ has bounded diameter. 
One then concludes that $\dgh(Y_k,Y_\infty) \ra 0$ and that $F :Y_\infty \ra X$ is 
isometric. 
\end{remark}

More precisely, we prove the following proposition.
\begin{proposition}\label{p4.1}
The set $F(\overset{\circ}{\cRe})$ satisfies, 
\begin{enumerate}
\item For any $x_1,x_2 \in F(\overset{\circ}{\cRe})$, 
$d_{g_0}(x_1,x_2) = \rho_\varepsilon(x_1,x_2)$.\\
\item $\overline{F(\overset{\circ}{\cRe} )} = X$.\\
\item  $F :(Y_\infty,d_\infty) \lra (X,d_{g_0})$ is an isometry.
\end{enumerate} \end{proposition}

\begin{proof} 
Let $x_1$,$x_2 \in F(\overset{\circ}{\cRe} )$. Without loss of generality, one can suppose that $x_2$ is not in the image of the cut-locus of $x_1$. Clearly,  $\rho_\varepsilon(x_1,x_2)\geq d_{g_0}(x_1,x_2)$. Let $\gamma:[0,1] \lra X$ be a $g_0$-minimal geodesic from $x_1$ to $x_2$. We do not know that $\gamma$ is in $F(\overset{\circ}{\cRe} )$ we then prove that there exist paths in $F(\overset{\circ}{\cRe} )$ arbitrarily close to $\gamma$. Let $r>0$ be a radius such that $B_{g_0}(x_2,r) \subset F(\overset{\circ}{\cRe} )$. We consider geodesics with the origin $x_1$ and the extremity in $B(x_2,\delta)$, for a small $\delta>0$. More precisely, let $u=\dot{\gamma}(0)$, then for any $v \in U_{x_1}X$ such that and $u\perp v$, one defines $\gamma_{s,v}(t) = exp_{x_1}(t(u+s.v)d(x_1,x_2))$. There exists $r(\delta)>0$ such that $\gamma_{s,v}(1) \in B(x_2,\delta)$ if $|s| \leq r(\delta)$ and one can choose $r(\delta)\ra 0$  as $\delta$ goes to $0$. 

We claim that for every $\delta>0$, there exists such $\gamma_{s,v}$ which 
is imbedded in $F(\overset{\circ}{\cRe} )$. 

Let us show that one can find such $\gamma_{s,v}$ disjoint from $F(\cS)$, where $\cS$ is the singular set of $Y_\infty$ defined in \ref{singular}. The idea is that if any 
$\gamma_{s,v}$ would hit $F(\cS)$ at least in one point, then the Hausdorff 
dimension of $F(\cS)$ would be larger than $n-1$, which is a contradiction. 
More precisely, one considers a truncated cone $U_{\delta}$ defined as follows. Let 
$$\Gamma : ]0,r(\delta)] \times (U_{x_1}X \cap u^\perp) \times [0,1] \rightarrow X$$ 
be defined by $\Gamma(s,v,t)=\gamma_{s,v}(t)$. If $\delta$ is sufficiently small, 
$\Gamma$ is an embedding. One defines $U_{\delta}=\Gamma(]0,r(\delta)] \times (U_{x_1}X \cap u^\perp) \times [0,1])$. Let us denote by 
$U_{\delta}(1/2)$ the hypersurface in $U_{\delta}$ defined as 
$\Gamma( ]0,r(\delta)]) \times (U_{x_1}X \cap u^\perp) \times \{1/2\})$. 
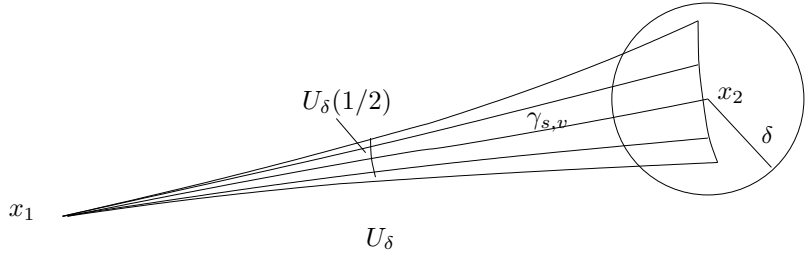
\begin{center}
\input{cone.pstex_t}
\end{center}
Let   $P : U_{\delta} \ra U_{\delta}(1/2)$ be the projection along geodesics 
defined by $P(\gamma_{s,v}(t)) = \gamma_{s,v}(1/2)$. Since we are on a fixed Riemannian manifold, there exists a constant $C>0$ 
such that  $P$ is $C$-lipschitz from $U_{\delta}$ to $X$. In particular, 
$P$ decreases the Hausdorff dimension,that is 
\begin{eqnarray*} 
 \dim_\h (P(U_{\delta} \cap F(\cS)))& \leq &  \dim_\h (U_\delta \cap F(\cS))\\
            &\leq & \dim_\h(\cS) \\
            & \leq & n-2\\
            & < & \dim U_{\delta}(1/2)=n-1.
\end{eqnarray*}
Hence, there exists $x \in U_{\delta}(1/2)$ such that $x \notin \Pi(F(\cS))$. 
This implies that the geodesic $\gamma_{s,v}$ such that $x =\gamma_{s,v}(1/2)$ does 
not intersect $F(\cS)$.  

We now prove that $\gamma_{s,v}$ is embedded in $F(\overset{\circ}{\cRe})$. 
Let $t_0 \in (0,1]$ be maximal such that $\gamma_{s,v}([0,t_0[) \subset 
F(\overset{\circ}{\cRe})$. By Lemma \ref{l4.4}, the path $\beta=F^{-1}\circ \gamma_{s,v}$ is well-defined on $[0,t_0[$ and  has a length bounded by $(1+c(\varepsilon))d(x_1,x_2)$. 
Since $F$ is bi-Lipschitz, $d_{g_k}(\beta (t), \beta (t'))\geq C|t'-t|$ and hence there exists a limit $p=\lim_{t \ra t_0} \beta(t) \in Y_\infty$. By continuity of $F$, $F(p)=\gamma_{s,v}(t_0)$  and  since $\gamma_{s,v}(t_0) \notin F(\cS)$ we have that $p \notin \cS$. This implies that  
$p \in \cR = \cap_{\varepsilon} \cRe = \cap_{\varepsilon>0} \overset{\circ}{\cRe}$ and consequently that $t_0=1$, because $ \overset{\circ}{\cRe}$ is open.
 
Hence 
\begin{eqnarray*}
\rho_\varepsilon(x_1,x_2)&\leq & \ell( \gamma_{s,v}) + d_0(\gamma_{s,v}(1),x_2)\\
 & \leq & \sqrt{1+r^2(\delta)}d_0(x_1,x_2) + \delta
\end{eqnarray*}
 As $\delta$ was arbitrary, this gives $\rho_\varepsilon(x_1,x_2) \leq d_0(x_1,x_2)$. 

The second assertion is proved in a similar way. Suppose there is a ball 
$B(x,r) \subset X\setminus F(\overset{\circ}{\cRe})$ and consider a geodesic $\gamma$ from a point 
$x_1$ inside $F(\overset{\circ}{\cRe})$ to $x$. Then we find another geodesic from $x_1$, 
close to $\gamma$, disjoint from $F(\cS)$ and with extremity in $X\setminus F(\overset{\circ}{\cRe})$. 
Arguing as above, we find a contradiction. 

Now 3) is straightforward. Using the density of $\overset{\circ}{\cRe} $ in $Y_{\infty}$ and of $F(\overset{\circ}{\cRe} )$ in $X$, we find that 
$F: (Y_\infty,d_{\infty}) \lra (X,d_0)$ is a $(1+c(\varepsilon))$-bi-Lipschitz homeomorphism 
for any $0 < \varepsilon < \varepsilon_0$ thus is isometric.
\end{proof}

\begin{proof}[End of Proof of theorem \ref{main-theorem}] 
Proposition \ref{p4.1} implies that the diameter of $(Y,g_k)$ remains 
bounded. Thus, $\dgh((Y,g_k),(Y_{\infty},d_{\infty})) \ra 0$ (for the non pointed convergence). As $(Y_{\infty},d_{\infty})$ is isometric to 
$(X;g_{0})$, one deduces that  $\dgh((Y,g_k),(X,g_0)) \ra 0$ as $k \ra \infty$. 
By theorem A.1.12 of \cite{Ch-Co}, $Y$ is diffeomorphic to $X$. The fact 
that $f$ is homotopic to a diffeomorphism is classic for hyperbolic 
manifolds.
\end{proof}


\end{document}

%% file: cone.pstex_t
\begin{picture}(0,0)%
\includegraphics{cone.pstex}%
\end{picture}%
\setlength{\unitlength}{2693sp}%
\begingroup\makeatletter\ifx\SetFigFont\undefined%
\gdef\SetFigFont#1#2#3#4#5{%
  \reset@font\fontsize{#1}{#2pt}%
  \fontfamily{#3}\fontseries{#4}\fontshape{#5}%
  \selectfont}%
\fi\endgroup%
\begin{picture}(7342,2337)(2461,-4297)
\put(9406,-3256){\makebox(0,0)[lb]{\smash{{\SetFigFont{9}{10.8}{\rmdefault}{\mddefault}{\updefault}$\delta$}}}}
\put(2476,-3931){\makebox(0,0)[lb]{\smash{{\SetFigFont{10}{12.0}{\rmdefault}{\mddefault}{\updefault}$x_1$}}}}
\put(9001,-2851){\makebox(0,0)[lb]{\smash{{\SetFigFont{10}{12.0}{\rmdefault}{\mddefault}{\updefault}$x_2$}}}}
\put(7246,-3076){\makebox(0,0)[lb]{\smash{{\SetFigFont{10}{12.0}{\rmdefault}{\mddefault}{\updefault}$\gamma_{s,v}$}}}}
\put(5176,-2941){\makebox(0,0)[lb]{\smash{{\SetFigFont{10}{12.0}{\rmdefault}{\mddefault}{\updefault}$U_{\delta}(1/2)$}}}}
\put(5761,-4201){\makebox(0,0)[lb]{\smash{{\SetFigFont{10}{12.0}{\rmdefault}{\mddefault}{\updefault}$U_{\delta}$}}}}
\end{picture}%